\documentclass[
a4paper,reqno]{amsart}
\usepackage[english]{babel}
\usepackage{amssymb,bbm,enumerate}
\usepackage{color}
\usepackage[linktocpage=true,colorlinks=true, linkcolor=blue, citecolor=red, urlcolor=green]{hyperref}
\usepackage[all]{xy}

\usepackage{float}
\restylefloat{table}

\usepackage{tikz}

\renewcommand{\ker}{\Ker}

\newcommand{\mc}[1]{\mathcal{#1}}
\newcommand{\mf}[1]{\mathfrak{#1}}
\newcommand{\mb}[1]{\mathbb{#1}}

\newcommand{\id}{\mathbbm{1}}

\DeclareMathOperator{\Mat}{Mat}
\DeclareMathOperator{\Hom}{Hom}
\DeclareMathOperator{\End}{End}

\DeclareMathOperator{\tr}{tr}

\DeclareMathOperator{\ad}{ad}
\DeclareMathOperator{\im}{Im}

\DeclareMathOperator{\Ker}{Ker}

\DeclareMathOperator{\gr}{gr}

\theoremstyle{plain}
\newtheorem{theorem}{Theorem}[section]
\newtheorem{lemma}[theorem]{Lemma}
\newtheorem{proposition}[theorem]{Proposition}
\newtheorem{corollary}[theorem]{Corollary}

\theoremstyle{definition}
\newtheorem{definition}[theorem]{Definition}
\newtheorem{example}[theorem]{Example}

\theoremstyle{remark}
\newtheorem{remark}[theorem]{Remark}

\numberwithin{equation}{section}

\definecolor{light}{gray}{.9}

\begin{document}

\title{A Lax type operator for quantum finite $W$-algebras}

\author{Alberto De Sole}
\address{Dipartimento di Matematica, Sapienza Universit\`a di Roma,
P.le Aldo Moro 2, 00185 Rome, Italy}
\email{desole@mat.uniroma1.it}
\urladdr{www1.mat.uniroma1.it/\$$\sim$\$desole}

\author{Victor G. Kac}
\address{Dept of Mathematics, MIT,
77 Massachusetts Avenue, Cambridge, MA 02139, USA}
\email{kac@math.mit.edu}

\author{Daniele Valeri}
\address{Yau Mathematical Sciences Center, Tsinghua University, 100084 Beijing, China}
\email{daniele@math.tsinghua.edu.cn}



\begin{abstract}
For a reductive Lie algebra $\mf g$, its nilpotent element $f$
and its faithful finite dimensional representation,
we construct a Lax operator $L(z)$
with coefficients in the quantum finite $W$-algebra $W(\mf g,f)$.
We show that for the classical linear Lie algebras 
$\mf{gl}_N$, $\mf{sl}_N$, $\mf{so}_N$ and $\mf{sp}_N$,
the operator $L(z)$ satisfies a generalized Yangian identity.
The operator $L(z)$ is a quantum finite analogue
of the operator of generalized Adler type
which we recently introduced in the classical affine setup.
As in the latter case, $L(z)$ is obtained as a generalized quasideterminant.
\end{abstract}

\keywords{
Quantum finite $\mc W$-algebra,
generalized quasideterminant,
twisted Yangians,
operators of twisted Yangian type,
Kazhdan filtration,
Rees algebra.
}

\maketitle

\tableofcontents

\section{Introduction}
\label{sec:1}

In our previous paper \cite{DSKV17},
for any nilpotent $N\times N$ matrix $f$
we introduced the following $r_1\times r_1$ matrix
with entries in $U(\mf g)((z^{-1}))$,
where $\mf g=\mf{gl}_N$ and $r_1$ is the multiplicity of the largest Jordan block of $f$:
\begin{equation}\label{vic:1.1}
\tilde{L}(z)
=
\big|z\id_N+f+\pi_{\leq\frac12}E+D\big|_{I_1J_1}
\,.
\end{equation}
Here $E=(e_{ji})_{i,j=1}^N$,
where $\{e_{ij}\}$ is the standard basis of the Lie algebra $\mf{gl}_N$,
$\pi_{\leq\frac12}$ is the projection on the $\leq\frac12$ part
of the $\ad x$-grading
\begin{equation}\label{vic:1.2}
\mf g
=
\oplus_{j\in\frac12\mb Z}\mf g_j
\,,
\end{equation}
associated to the $\mf{sl}_2$-triple $(f,2x,e)$,
$I_1\in\Mat_{N\times r_1}\mb F$, $J_1\in\Mat_{r_1\times N}\mb F$,
$D$ is a certain $N\times N$ diagonal matrix over $\mb F$, called the \emph{shift matrix},
and $|A|_{I_1J_1}=(J_1A^{-1}I_1)^{-1}$ denotes the $(I_1,J_1)$-quasideterminant 
of the invertible $N\times N$ matrix $A$ (over a unital associative algebra),
see \cite[Sec.4.1-4.3]{DSKV17} for details.
A discussion of generalized quasideterminants can be found in \cite[Sec.2.2]{DSKV17}.
In Section 2 of the present paper we give a coordinate free definition.

Let $I$ be the left ideal of $U(\mf g)$ generated by the set $\{m-(f|m)\,|\,m\in\mf g_{\geq1}\}$,
where $(\cdot\,|\,\cdot)$ is the trace form on $\mf g$,
and let $M=U(\mf g)/I$ be the corresponding left $U(\mf g)$-module;
it is also naturally a right $\mf g_{\geq\frac12}$-module.
Then, the subspace
$$
W(\mf g,f)
=
M^{\ad\mf g_{\geq\frac12}}
\,,
$$
has a natural structure of a unital associative algebra induced from that of $U(\mf g)$,
which is called the \emph{quantum finite} $W$-\emph{algebra} associated to 
the nilpotent element $f$ of $\mf g$.

Let
\begin{equation}\label{vic:1.3}
L(z)
=
\tilde{L}(z)\bar 1
\in\Mat_{r_1\times r_1}M((z^{-1}))
\,.
\end{equation}
The first main theorem of \cite{DSKV17} (Theorem 4.2) states that the matrix $L(z)$
has entries with coefficients in $W(\mf g, f)$,
and the second main theorem of \cite{DSKV17} (Theorem 4.3)
states that $L(z)$ is an operator of  Yangian type
for the algebra $W(\mf g,f)$.

Recall that for a unital associative algebra $R$ and a vector space $V$
the \emph{Yangian identity} for $A(z)\in R((z^{-1}))\otimes\End V$
is the following identity in $R[[z^{-1},w^{-1}]][z,w]\otimes\End V\otimes\End V$:
\begin{equation}\label{vic:1.4}
(z-w)[A(z),A(w)]
=
\Omega_V(A(w)\otimes A(z)-A(z)\otimes A(w))
\,,
\end{equation}
where $\Omega_V\in\End V\otimes\End V$
is the operator of permutation of factors.
This identity appeared in the famous talk of Drinfeld \cite{Dr86}
in the definition of the Yangian of $\mf{gl}_N$.
An operator $A(z)$ satisfying identity \eqref{vic:1.4}
is called an operator of \emph{Yangian type}.

Note that the Yangian identity is the finite quantum counterpart
of the Adler identity from the theory of classical affine $W$-algebras,
introduced in \cite{DSKV15},
building on the work of Adler \cite{Adl79},
and developed in our papers \cite{DSKV16a,DSKV16b}.
Note also that the mysterious shift matrix $D$ 
is a purely quantum effect, which does not appear in the classical situation.

The first main theorem of the present paper (Theorem \ref{thm:main1}) is a far reaching
generalization of the first main theorem of \cite{DSKV17}.
Namely, we replace $\mf{gl}_N$ by an arbitrary reductive Lie algebra $\mf g$,
and the standard representation of $\mf{gl}_N$ in $\mb F^N$
by an arbitrary faithful representation $\varphi$ of $\mf g$ 
in a finite dimensional vector space $V$.
We assume in addition that the trace form $(\cdot\,|\,\cdot)_V$
is non-degenerate (which automatically holds if $\mf g$ is semisimple).
To these data and a nilpotent element $f$ of $\mf g$,
we associate an analogue of the operator $\tilde{L}(z)$
defined by \eqref{vic:1.1},
and of the operator $L(z)$, defined by \eqref{vic:1.3},
by replacing the operators $E$ and $D$
by the operators $E_{\mf g,V}$ and $D_{\mf g,V}$ defined as follows.
Choose a basis $\{u_i\}_{i\in I}$ of $\mf g$
compatible with the grading \eqref{vic:1.2},
and let $\{u^i\}_{i\in I}$ be the dual basis of $\mf g$ w.r.t. the trace form.
Let
\begin{equation}\label{vic:1.5}
E_{\mf g,V}
=
\sum_{i\in I} u_i\otimes \varphi(u^i)
\in U(\mf g)\otimes\End V
\,,
\end{equation}
and
\begin{equation}\label{vic:1.6}
D_{\mf g,V}
=
-\sum_{i\in I, u_i\in\mf g_{\geq1}} \varphi(u^i)\varphi(u_i)
\in \End V
\,.
\end{equation}
Note that the shift operator $D$ of \cite{DSKV17}
was constructed by a rather complicated combinatorial procedure,
but it is easy to see that it coincides with $D_{\mf{gl}_N,\mb F^N}$.

Theorem \ref{thm:main1} states that the $r_1\times r_1$-matrix
$L(z)$ has entries with coefficients in $W(\mf g,f)$,
where $r_1$ is the dimension of the $\varphi(x)$-eigenspace in $V$
attached to the maximal eigenvalue.

Unfortunately, an analogue of the second main theorem of \cite{DSKV17}
does not appear to hold in such a generality.
In fact, for the second main theorem 
we need to assume that $\mf g$ is one of the classical Lie algebras 
$\mf{gl}_N$, $\mf{sl}_N$, $\mf{so}_N$ or $\mf{sp}_N$,
and $\varphi$ is its standard representation in $\mb F^N$.

We found that the Yangian identity \eqref{vic:1.4} for $\mf{gl}_N$,
which also holds for $\mf{sl}_N$,
should be generalized to the following $(\alpha,\beta,\gamma)$-\emph{Yangian identity}:
\begin{equation}\label{vic:1.7}
\begin{array}{l}
\displaystyle{
\vphantom{\Big(}
(z-w+\alpha\Omega_V)
(A(z)\otimes\id_V)
(z+w+\gamma-\beta\Omega_V^\dagger)
(\id_V\otimes A(w))
} \\
\displaystyle{
\vphantom{\Big(}
=
(\id_V\otimes A(w))
(z+w+\gamma-\beta\Omega_V^\dagger)
(A(z)\otimes\id_V)
(z-w+\alpha\Omega_V)
\,,}
\end{array}
\end{equation}
where $\alpha,\beta,\gamma\in\mb F$.
Here, if $\beta\neq0$,
we assume that $V$ is endowed with a non-degenerate symmetric or skewsymmetric bilinear form,
and $\Omega_V^\dagger$ is obtained by taking the adjoint (w.r.t. this form) of the first factor in $\Omega_V$.
Note that for $\alpha=1$, $\beta=\gamma=0$,
identity \eqref{vic:1.7} turns into \eqref{vic:1.4},
while for $\alpha=\beta=-1$, $\gamma=0$,
it turns into the RSRS presentation of the extended twisted Yangian 
of $\mf{so}_N$ and $\mf{sp}_N$, introduced by Olshanski \cite{Ols92}, see also \cite{Mol07}.

Our second main theorem (Theorem \ref{thm:main2})
states that for $\mf g=\mf{gl}_N$, $\mf{sl}_N$, $\mf{so}_N$ or $\mf{sp}_N$,
and $V=\mb F^N$,
the operator $L(z)$ satisfies the $(\alpha,\beta,\gamma)$-Yangian identity
with $(\alpha,\beta,\gamma)=(1,0,0)$ for $\mf g=\mf{gl}_N$ or $\mf{sl}_N$,
and $(\alpha,\beta,\gamma)=(\frac12,\frac12,\frac{\epsilon-N+r_1}2)$,
where $\epsilon=+1$, (espectively $-1$),
if $\mf g=\mf {so}_N$ (resp. $\mf {sp}_N$).

The classical affine analogues of both the main Theorems \ref{thm:main1} and \ref{thm:main2}
have been established in our recent paper \cite{DSKV18}.
These results led, in the context of classical affine $W$-algebras,
to a large class of integrable hierarchies
of Hamiltonian equations of Lax type,
encompassing all Drinfeld-Sokolov hierarchies attached 
to classical affine Lie algebras \cite{DS85}.

Throughout the paper the base field $\mb F$ is a field of characteristic zero.

\subsubsection*{Acknowledgments} 

We are grateful to the referee for careful reading of the paper,
for several corrections and various enlightening observations.
The first author would like to acknowledge
the hospitality of MIT, where he was hosted during the spring semester of 2017,
when this project started.
The second author would like to acknowledge
the hospitality of the University of Rome La Sapienza
during his visit in Rome in January 2017.
The third author is grateful to the University of Rome La Sapienza
for its hospitality during his several visits in 2016 and 2017.
All three authors are extremely grateful to IHES
for their kind hospitality during the summer of 2017,
when the paper was completed.
The first author is supported by National FIRB grant RBFR12RA9W,
National PRIN grant 2015ZWST2C,
and University grant C26A158K8A,
the second author was supported by an NSF grant.

\section{Dirac reduction and generalized quasideterminants in linear algebra}
\label{sec:2}

Let $R$ be a unital associative algebra over $\mb F$
and let 
\begin{equation}\label{eq:linalg1}
\chi_\alpha\,:\,\,
0\to U_\alpha
\stackrel{\Psi_\alpha}{\longrightarrow} 
V_\alpha
\stackrel{\Pi_\alpha}{\longrightarrow} 
W_\alpha\to0
\,,\,\,
\alpha=1,2
\,,
\end{equation}
be two short exact sequences of $R$-modules.

\subsection{Dirac reduction}
\label{sec:2.1}

Let $A:\,V_1\to V_2$ be an $R$-module endomorphism.
If the following conditions are met:
\begin{equation}
\label{eq:linalg2}
\text{(i) }\,
\im\Psi_1=\ker\Pi_1\subset\ker A
\,\,\,,\qquad\qquad\qquad
\text{ (ii) }\,
\im A\subset \im\Psi_2=\ker\Pi_2
\,,
\end{equation}
then, 
we have the canonically induced $R$-module homomorphism
$\bar A:\,W_1\to U_2$ making the following diagram commute:
\begin{equation}\label{eq:linalg3}
\UseTips
\xymatrix{
V_1
\ar[d]_{\Pi_1}
\ar[rr]^{A}
& &
V_2  \\
W_1
\ar[rr]_{\bar A\,=\,\Psi_2^{-1}A\Pi_1^{-1}}
& &
\ar[u]^{\Psi_2}
U_2
}
\end{equation}
If, on the other hand, conditions (i) and (ii) are not met,
we can still induce a well defined $R$-module 
homomorphism $W_1\to U_2$,
at the price of ``Dirac modifying'' the endomorphism $A$.
This can be done provided that
\begin{equation}\label{eq:linalg4}
\Pi_2 A\Psi_1\,:\,\, U_1\longrightarrow W_2
\,\,\text{ is an isomorphism }
\,.
\end{equation}
In this case, we define the \emph{Dirac modified} $R$-module homomorphism
\begin{equation}\label{eq:linalg5}
A^D_{\chi_1,\chi_2}
:=
A-A\Psi_1(\Pi_2 A\Psi_1)^{-1}\Pi_2 A
\,\in\Hom_R(V_1,V_2)
\,.
\end{equation}
\begin{lemma}\label{lem:linalg1}
Assume that condition \eqref{eq:linalg4} holds.
Then, the Dirac modified homomorphism $A^D_{\chi_1,\chi_2}:\,V_1\to V_2$
is well defined and it satisfies conditions (i) and (ii) in \eqref{eq:linalg2}.
Hence, we get an induced $R$-module homomorphism
\begin{equation}\label{eq:linalg6}
\bar A^D_{\chi_1,\chi_2}
=
\Psi_2^{-1}\big(A-A\Psi_1(\Pi_2 A\Psi_1)^{-1}\Pi_2 A\big)\Pi_1^{-1}
\,:\,\,W_1\to U_2
\,,
\end{equation}
that we call the \emph{Dirac reduction} of $A$ w.r.t. the short exact sequences $\chi_1$ and $\chi_2$.
\end{lemma}
\begin{proof}
Conditions (i) and (ii) are equivalent, respectively, to the equations
$$
A^D_{\chi_1,\chi_2}\Psi_1=0
\,\,\text{ and }\,\,
\Pi_2A^D_{\chi_1,\chi_2}=0\,,
$$
which can be immediately checked.
\end{proof}
\begin{remark}\label{rem:linalg2}
Let $\mc V$ be a Poisson algebra.
Recall that, 
given a set of elements $\theta_1,\dots,\theta_r\in\mc V$ (constraints)
one defines the Dirac reduced Poisson algebra structure 
on the algebra $\mc V/\langle\theta_i\rangle_{i=1}^r$
by a well defined Poisson bracket
\begin{equation}\label{eq:linalg9}
\{f,g\}^D
=
\{f,g\}-\sum_{i,j=1}^r\{f,\theta_i\}(S^{-1})_{ij}\{\theta_j,g\}
\,,
\end{equation}
where $S$ is the matrix with entries
$S_{ij}=\{\theta_i,\theta_j\}$,
and is assumed to be invertible.
In terms of the corresponding Poisson structure $H:\,\mc V^\ell\to\mc V^\ell$
($\ell$ being the number of independent variables)
formula \eqref{eq:linalg9} becomes
\begin{equation}\label{eq:linalg9b}
H^D(F)
=
H(F)-\sum_{i,j=1}^r H(\nabla\theta_i)(S^{-1})_{ij}\nabla\theta_j\cdot H(F)
\,,
\end{equation}
which is a special case of \eqref{eq:linalg5}
(with the following data: $V_1=\mc V\otimes\mc V^\ell$;
$V_2=\mc V^\ell$;
$A(g\otimes F)=gH(F)$;
$\Psi_1:\,\mc V^r\to\mc V\otimes\mc V^\ell,\,\big(g_i\big)_{i=1}^r\mapsto\sum_ig_i\otimes\nabla\theta_i$;
$\Pi_2:\,\mc V^\ell\to \mc V^r,\,F\to\big(F\cdot\nabla\theta_i\big)_{i=1}^r$).
This is the reason for naming $\bar A^D_{\chi_1,\chi_2}$ the ``Dirac reduction'' of $A$.
\end{remark}

\subsection{Generalized quasideterminant}
\label{sec:2.2}

The \emph{generalized quasideterminant} 
of the $R$-module homomorphism $A:\,V_1\to V_2$
with respect to the the maps $\Psi_2$ and $\Pi_1$ in \eqref{eq:linalg1}, 
is the $R$-module homomorphism (cf. \cite{DSKV16a})
\begin{equation}\label{eq:linalg7}
|A|_{\Psi_2,\Pi_1}
:=
\big(\Pi_1 A^{-1}\Psi_2\big)^{-1}
\,:\,\,W_1\to U_2
\,,
\end{equation}
provided that it exists, i.e. provided that $A:\,V_1\to V_2$ is invertible,
and that 
\begin{equation}\label{eq:linalg8}
\Pi_1 A^{-1}\Psi_2:\,U_2\to W_1
\,\,\text{ is invertible }
\,.
\end{equation}
\begin{remark}\label{rem:linalg3}
Recall that, given a square matrix $A\in\Mat_{N\times N}R$,
its $(i,j)$-quasi-determinant, if it exists, is defined as \cite{GGRW05}
$$
|A|_{ij}=\big((A^{-1})_{ji}\big)^{-1}
\,\in R
\,.
$$
In \cite{DSKV16a} we generalized this notion as follows:
given rectangular matrices $I\in\Mat_{N\times M}\mb F$ and $J\in\Mat_{M\times N}\mb F$,
the $(I,J)$-quasideterminant of $A$, if it exists, is defined as
$$
|A|_{IJ}=\big(JA^{-1}I\big)^{-1}
\,\in\Mat_{M\times M} R
\,.
$$
Obviously, \eqref{eq:linalg7} provides a further generalization of the notion 
of quasideterminant, hence the name ``generalized quasideterminant''.
\end{remark}

\subsection{Dirac reduction as a generalized quasideterminant}
\label{sec:2.3}

\begin{proposition}\label{prop:linalg}
Suppose that the $R$-module homomorphism $A:\,V_1\to V_2$
is invertible.
Then, 
the Dirac reduction $\bar A^D_{\chi_1,\chi_2}:\,W_1\to U_2$ exists,
i.e. \eqref{eq:linalg4} holds,
if and only if the generalized quasideterminant $|A|_{\Psi_2,\Pi_1}:\,W_1\to U_2$ exists,
i.e. \eqref{eq:linalg8} holds,
and, in this case, they coincide: $\bar A^D_{\chi_1,\chi_2}=|A|_{\Psi_2,\Pi_1}$.
\end{proposition}
\begin{proof}
Since $A$ is invertible, 
to say that $\Pi_2A\Psi_1:\, U_1\to W_2$ is invertible is equivalent to the conditions
$$
A(\im\Psi_1)\cap\im\Psi_2=0 
\,\,\text{ and }\,\, 
V_2=A(\im\Psi_1)+\im\Psi_2
\,.
$$
Applying $A^{-1}$ to both these equalities, we get
$$
A^{-1}(\im\Psi_2)\cap\im\Psi_1=0 \,\,\text{ and }\,\, V_1=A^{-1}(\im\Psi_2)+\im\Psi_1
\,,
$$
which is equivalent to saying that $\Pi_1A^{-1}\Psi_2:\, U_2\to W_1$ is invertible.
This proves the first statement.
We are left to prove the equation $\bar A^D_{\chi_1,\chi_2}=|A|_{\Psi_2,\Pi_1}$.
By definition, we have
$$
\begin{array}{l}
\displaystyle{
\vphantom{\Big(}
\Pi_1A^{-1}\Psi_2\bar A^D_{\chi_1,\chi_2}
=
\Pi_1A^{-1}\Psi_2
\Psi_2^{-1}\big(A-A\Psi_1(\Pi_2 A\Psi_1)^{-1}\Pi_2 A\big)\Pi_1^{-1}
} \\
\displaystyle{
\vphantom{\Big(}
=
\Pi_1\big(\id_V-\Psi_1(\Pi_2 A\Psi_1)^{-1}\Pi_2 A\big)\Pi_1^{-1}
=
\Pi_1\Pi_1^{-1}
=\id_{W_1}
\,,}
\end{array}
$$
since $\Pi_1\Psi_1=0$.
Hence, $\bar A^D_{\chi_1,\chi_2}$ is a right inverse of $\Pi_1A^{-1}\Psi_2$.
A similar computation shows that 
$\bar A^D_{\chi_1,\chi_2}$ is a left inverse of $\Pi_1A^{-1}\Psi_2$ as well,
proving the claim.
\end{proof}

\section{Review of finite \texorpdfstring{$W$}{W}-algebras}
\label{sec:3}

Let $\mf g$ be a reductive Lie algebra with a non-degenerate symmetric 
invariant bilinear form $(\cdot\,|\,\cdot)$.
Let $f\in\mf g$ be a nilpotent element;
by the Jacobson-Morozov Theorem, it can be included in an $\mf{sl}_2$-triple
$\{f,2x,e\}\subset\mf g$.
We have the corresponding $\ad x$-eigenspace decomposition
\begin{equation}\label{eq:grading}
\mf g=\bigoplus_{k\in\frac{1}{2}\mb Z}\mf g_{k}
\,\,\text{ where }\,\,
\mf g_k=\big\{a\in\mf g\,\big|\,[x,a]=ka\big\}
\,,
\end{equation}
so that $f\in\mf g_{-1}$, $x\in\mf g_{0}$ and $e\in\mf g_{1}$.
We shall denote, for $j\in\frac12\mb Z$, $\mf g_{\geq j}=\oplus_{k\geq j}\mf g_k$,
and similarly $\mf g_{\leq j}$.

A key role in the theory of $W$-algebras is played by the left ideal
\begin{equation}\label{0225:eq4}
J=
U(\mf g)\big\langle m-(f|m) \big\rangle_{m\in\mf g_{\geq1}}
\subset U(\mf g)
\,,
\end{equation}
and the corresponding left $\mf g$-module 
\begin{equation}\label{eq:M}
M=U(\mf g)/J\,.
\end{equation}
We shall denote by $\bar 1\in M$ the image of $1\in U(\mf g)$
in the quotient space.
Note that, by definition, $g\bar1=0$ if and only if $g\in J$.
\begin{lemma}[{\cite[Lem.3.1]{DSKV17}}]\label{0303:lem4}
\begin{enumerate}[(a)]
\item
$U(\mf g)JU(\mf g_{\geq\frac12})\subset J$.
\item
The Lie algebra $\mf g$ acts on the module $M$ by left multiplication,
and its subalgebra $\mf g_{\geq\frac12}$ 
acts on $M$ both by left and by right multiplication
(hence, also via adjoint action).
\end{enumerate}
\end{lemma}
Consider the subspace
\begin{equation}\label{0225:eq3}
\widetilde{W}
:=
\big\{
w\in U(\mf g)\,\big|\,
[a,w]\bar 1=0\,\text{ in }\, M\,,\,\,\text{for all } a\in\mf g_{\geq\frac12}
\big\}
\,\subset U(\mf g)\,.
\end{equation}
\begin{lemma}[{\cite[Lem.3.2]{DSKV17}}]\phantomsection\label{0225:prop1a}
\begin{enumerate}[(a)]
\item
$J\subset\widetilde{W}$.
\item
For $h\in J$ and $w\in\widetilde{W}$, we have $hw\in J$.
\item
$\widetilde{W}$ 
is a subalgebra of $U(\mf g)$.
\item
$J$ is a (proper) two-sided ideal of $\widetilde{W}$.
\end{enumerate}
\end{lemma}
\begin{proposition}[{\cite[Prop.3.3]{DSKV17}}]\label{0225:prop1}
The quotient
\begin{equation}\label{0225:eq8}
W(\mf g,f)
=
M^{\ad\mf g_{\geq\frac12}}
=
\widetilde{W}/J
\end{equation}
has a natural structure of a unital associative algebra,
induced by that of $U(\mf g)$.
\end{proposition}
\begin{definition}\label{def:Walg}
The \emph{finite} $W$-\emph{algebra}
associated to the Lie algebra $\mf g$ and its nilpotent element $f$
is the algebra $W(\mf g,f)$ defined in \eqref{0225:eq8}.
\end{definition}

\section{The operator \texorpdfstring{$L(z)$}{L(z)} for the \texorpdfstring{$W$}{W}-algebra \texorpdfstring{$W(\mf g,f)$}{W(g,f)}}
\label{sec:4}

\subsection{Setup and notation}\label{sec:4.1}

As in Section \ref{sec:3},
let $\mf g$ be a reductive Lie algebra,
let $\{f,2x,e\}\subset\mf g$ be an $\mf{sl}_2$-triple
and let \eqref{eq:grading}
be the corresponding $\ad x$-eigenspace decomposition.

Let $\varphi:\,\mf g\to\End V$ be a faithful representation of $\mf g$
on the $N$-dimensional vector space $V$.
Throughout the paper we shall often use the following convention:
we denote by lowercase latin letters elements of the Lie algebra $\mf g$,
and by the same uppercase letters the corresponding (via $\varphi$)
elements of $\End V$.
For example, $F=\varphi(f)$ is a nilpotent endomorphism of $V$.
Moreover, $X=\varphi(x)$ is a semisimple endomorphism of $V$ with half-integer eigenvalues.
The corresponding $X$-eigenspace decomposition of $V$ is
\begin{equation}\label{eq:grading_V}
V=\bigoplus_{k\in\frac12\mb Z}V[k]
\,.
\end{equation}
Let $\frac d2$ be the largest $X$-eigenvalue in $V$.
We also have the corresponding $\ad X$-eigenspace 
decomposition of $\End V$:
\begin{equation}\label{eq:grading_EndV}
\End V=\bigoplus_{k\in\frac12\mb Z}(\End V)[k]
\,,
\end{equation}
which has largest eigenvalue $d$.
We shall denote, for $k\in\frac12\mb Z$, $V[\geq k]=\oplus_{j\geq k}V[j]$,
and similarly $V[>k],\,V[\leq k],\,V[<k]$.
Also, we shall denote $(\End V)[\geq k]=\oplus_{j\geq k}(\End V)[j]$,
and similarly for $(\End V)[\leq k]$, etc.

We shall denote, for $k\in\frac12\mb Z$, the maps
\begin{equation}\label{eq:chi}
\Psi_k:\,V[k]\hookrightarrow V
\,\text{ and }\,
\Pi_k:\,V\twoheadrightarrow V[k]
\,,
\end{equation}
where $\Psi_k$ is the natural immersion and $\Pi_k$ is the projection
w.r.t. the decomposition \eqref{eq:grading_V}.
Similarly, we shall also denote
\begin{equation}\label{eq:chi2}
\begin{array}{l}
\displaystyle{
\vphantom{\Big(}
\Psi_{>k}:\,V[>k]\hookrightarrow V
\,,\,\,
\Psi_{<k}:\,V[<k]\hookrightarrow V
\,,} \\
\displaystyle{
\vphantom{\Big(}
\Pi_{>k}:\,V\twoheadrightarrow V[>k]
\,,\,\
\Pi_{<k}:\,V\twoheadrightarrow V[<k]
\,.}
\end{array}
\end{equation}
Using these maps, we can construct the short exact sequences
\begin{equation}\label{eq:chi3}
\begin{array}{l}
\displaystyle{
\vphantom{\Big(}
\chi_1\,:\,\,
0\to V\big[>-\frac d2\big]
\stackrel{\Psi_{>-\frac d2}}{\longrightarrow} 
V
\stackrel{\Pi_{-\frac d2}}{\longrightarrow} 
V\big[-\frac d2\big]
\,,} \\
\displaystyle{
\vphantom{\Big(}
\chi_2\,:\,\,
0\to V\big[\frac d2\big]
\stackrel{\Psi_{\frac d2}}{\longrightarrow} 
V
\stackrel{\Pi_{<\frac d2}}{\longrightarrow} 
V\big[<\frac d2\big]
\,.}
\end{array}
\end{equation}

Recalling the $\ad x$-eigenspace decomposition \eqref{eq:grading}
we shall denote, for $k\in\frac12\mb Z$,
\begin{equation}\label{20170623:eq1}
\pi_k\,:\,\,
\mf g\twoheadrightarrow \mf g_k
\,,
\end{equation}
the projection w.r.t. \eqref{eq:grading},
and similarly for the maps
\begin{equation}\label{20170623:eq2}
\pi_{>k}:\,
\mf g\twoheadrightarrow \mf g_{>k}
\,\,,\,\,\,\,
\pi_{<k}\,:\,\,
\mf g\twoheadrightarrow \mf g_{<k}
\,\,,\,\,\,\,
\pi_{\geq k}\,:\,\,
\mf g\twoheadrightarrow \mf g_{\geq k}
\,\,,\,\,\,\,
\pi_{\leq k}\,:\,\,
\mf g\twoheadrightarrow \mf g_{\leq k}
\,.
\end{equation}
We shall also denote, with a slight abuse of notation,
\begin{equation}\label{20170623:eq1b}
\Pi_k\,:\,\,
\End V\twoheadrightarrow (\End V)[k]
\,,
\end{equation}
the projection with respect to the $\ad X$-eigenspace decomposition \eqref{eq:grading_EndV},
and similarly for $\Pi_{>k},\,\Pi_{<k},\,\Pi_{\geq k},\,\Pi_{\leq k}$.

Recall that the trace form on $\mf g$ associated to the representation $V$ 
is, by definition,
\begin{equation}\label{20170317:eq1}
(a|b)=\tr_V(\varphi(a)\varphi(b))\,,
\qquad
a,b\in\mf g
\,,
\end{equation}
and we assume that it is non-degenerate.
Let $\{u_i\}_{i\in I}$ be a basis of $\mf g$ 
compatible with the $\ad x$-eigenspace decomposition \eqref{eq:grading},
i.e. $I=\sqcup_k I_k$ where $\{u_i\}_{i\in I_k}$ is a basis of $\mf g_k$.
We also denote $I_{\leq\frac12}=\sqcup_{k\leq\frac12}I_k$,
and similarly for $I_{\leq0}$, $I_{\geq\frac12}$, etc.
Let $\{u^i\}_{i\in I}$ be the basis of $\mf g$ dual to $\{u_i\}_{i\in I}$ with respect 
to the form \eqref{20170317:eq1},
i.e. $(u_i|u^j)=\delta_{i,j}$.
According to our convention,
we denote by $U_i=\varphi(u_i)$ and $U^i=\varphi(u^i)$, $i\in I$, 
the corresponding endomorphisms of $V$.

Consider the following important element
\begin{equation}\label{20170623:eq4}
U=\sum_{i\in I}u_i U^i\in\mf g\otimes\End V
\,.
\end{equation}
Here and further we are omitting the tensor product sign.
Then we have, according to the notation \eqref{20170623:eq1}-\eqref{20170623:eq1b}, the identities
\begin{equation}\label{20170623:eq3}
\pi_k U=\Pi_{-k} U
\,\,,\,\,\,\,
\pi_{\geq k} U=\Pi_{\leq -k} U
\,,\dots
\,,
\end{equation}
where $\pi_k$ and $\pi_{\geq k}$
act on the first factors of the tensor product $\mf g\otimes\End V$,
while $\Pi_{-k}$ and $\Pi_{\leq -k}$ act on the second factors.

We shall denote by $\delta(a)$ the eigenvalue of $\ad x$ on a (homogeneous) element $a\in\mf g$,
i.e.
\begin{equation}\label{20170623:eq5}
\delta(a)=k
\,\,\text{ if and only if }\,\,
a\in\mf g_k
\,.
\end{equation}
Similarly, for an index $i\in I$, we shall denote
\begin{equation}\label{20170623:eq6}
\delta(i)=\delta(u_i)
\,.
\end{equation}
Throughout the paper, we shall use the following convenient notation
on summations ($h,k\in\frac12\mb Z$):
\begin{equation}\label{20170623:eq7}
\sum_{h\leq\delta(i)\leq k}F(i)
=
\sum_{h\leq\ell\leq k}\sum_{i\in I_\ell}F(i)
\,,
\end{equation}
where $F(i)$ is any expression depending on the index $i$.

\subsection{The shift matrix}\label{sec:4.1b}

The following endomorphism of $V$ 
(which we will call the ``shift matrix'')
will play an important role in the paper:
\begin{equation}\label{eq:D}
D
=
-\sum_{\delta(i)\geq1}U^iU_i
\,\in(\End V)[0]
\,,
\end{equation}
where we are using the notation \eqref{20170623:eq7}.
Note that, by definition, we have (cf. \eqref{eq:chi}):
\begin{equation}\label{20170623:eq8}
\Pi_{\frac d2}D=0
\,\,\text{ and }\,\,
D\Psi_{\frac d2}=0
\,.
\end{equation}
\begin{remark}\label{20170628:rem}
Note that $D$ remains unchanged if we replace $U_i$ by $\varphi(v_i)$,
where $\{v_i\}$ is any basis of $\mf g_{\geq1}$,
and we replace $U^i$ by $\varphi(v^i)$,
where $\{v^i\}$ is the dual (w.r.t. the trace form of $\mf g$) basis of $\mf g_{\leq-1}$.
Moreover, $D$ commutes with the action of $\varphi(\mf g_0)$ on $V$.
In particular, $D$ commutes with any Cartan subalgebra of $\mf g$ contained in $\mf g_0$,
and therefore it preserves the corresponding weight space decomposition of $V$.
However, as Example \ref{ex:Dadg} below shows, 
$D$ does not act necessarily as a scalar on each weight space.
\end{remark}
\begin{example}\label{ex:DglN}
Consider the Lie algebras $\mf{gl}_N$ or $\mf{sl}_N$ and their standard representation $V=\mb F^N$.
In this case, it is not hard to compute the shift matrix $D$ explicitly
(for example, by fixing the standard basis $\{E_{ij}\}$ of elementary matrices,
and its dual, w.r.t. the trace form, basis $\{E_{ji}\}$,
and assuming that the degree operator $X$ is diagonal).
As a result, we get (both for $\mf{gl}_N$ and $\mf{sl}_N$):
\begin{equation}\label{20170623:eq11}
D
=
-\sum_{k\in\frac12\mb Z}\dim(V[\geq k+1]) \id_{V[k]}
\,,
\end{equation}
which is the same as the diagonal matrix defined in \cite[Eq.(4.6)]{DSKV17}.
As a side remark, taking the trace of both sides of \eqref{20170623:eq11},
we get the following interesting combinatorial identity:
$$
\dim(\mf g_{\geq1})=\sum_{k}\dim(V[k])\dim(V[\geq k+1])
\,.
$$
\end{example}
\begin{example}\label{ex:DsoN}
Consider the Lie algebra $\mf{so}_N$
and its standard representation $V=\mb F^N$.
With a similar computation as in Example \ref{ex:DglN}
(for example, one can represent the Lie algebra $\mf{so}_N$ 
as the subalgebra of $\mf{gl}_N$
spanned by the matrices 
$F_{ij}=E_{ij}-E_{N+1-j,N+1-i}$,
and take the basis $\{F_{ij}\}_{j<N+1-i}$,
and the dual basis $\{\frac12F_{ji}\}$), we get
\begin{equation}\label{20170623:eq12}
D
=
-\frac12\sum_{k\in\frac12\mb Z}\dim(V[\geq k+1]) \id_{V[k]}+\frac12\id_{V[\leq-\frac12]}
\,.
\end{equation}
\end{example}
\begin{example}\label{ex:DspN}
Consider the Lie algebra $\mf{sp}_N$, for even $N$,
and its standard representation $V=\mb F^N$.
The computation in this case is analogous to that of Example \ref{ex:DsoN}.
In this case we take the basis 
$F_{ij}=E_{ij}-(-1)^{i+j}E_{N+1-j,N+1-i}$,
and the indices satisfy the inequality $j\leq N+1-i$.
The result for the shift matrix in this case is
\begin{equation}\label{20170623:eq13}
D
=
-\frac12\sum_{k\in\frac12\mb Z}\dim(V[\geq k+1]) \id_{V[k]}-\frac12\id_{V[\leq-\frac12]}
\,.
\end{equation}
\end{example}
\begin{remark}
Note that \eqref{20170623:eq12} and \eqref{20170623:eq13}
are, up to a scaling factor and adding an overall constant, the same shifts that they get in \cite{Br09}
to describe the finite $W$-algebras associated to the Lie algebras 
$\mf{so}_N$ and $\mf{sp}_N$ and their rectangular nilpotent elements corresponding to a partition of the form $N=p+\dots+p$
with $p$ even (for $p$ odd the role of the shift matrix $D$ and the shifts in \cite{Br09} is different).
Hence, our results of Section \ref{sec:5}
extend, in particular, the construction of \cite{Br09}
to arbitrary nilpotent elements of $\mf{so}_N$ and $\mf{sp}_N$.
\end{remark}
\begin{example}\label{ex:Dsl2}
Consider the Lie algebra $\mf{sl}_2$ and its irreducible representation 
in the $N$-dimensional vector space $V$.
The action of $\mf{sl}_2$ in some basis of $V$ is given by
$$
F=\sum_{i=1}^{N-1}E_{i+1,i}
\,\,,\,\,\,\,
X=\frac12\sum_{i=1}^N(N+1-2i)E_{ii}
\,\,,\,\,\,\,
E=\sum_{i=1}^{N-1}i(N-i)E_{i,i+1}
\,.
$$
In this case, 
the shift matrix \eqref{eq:D} is easily computed:
$$
D
=
-\frac1{\tr(FE)}FE
=
-\frac6{N^3-N}\sum_{i=1}^N(i-1)(N+1-i) E_{ii}
\,.
$$

\end{example}
\begin{example}\label{ex:Dadg}
Let $V=\mf g$ be the adjoint representation of $\mf g$ and let $f\in\mf g$ be a principal nilpotent element.
In this case, $\mf g_0$ is a Cartan subalgebra of $\mf g$,
and consider the corresponding root space decomposition 
$\mf g=\mf g_{\leq-1}\oplus\mf g_0\oplus\mf g_{\geq1}$,
with $\mf g_{\geq1}=\oplus_{\alpha>0}\mb Fe_\alpha$
and $\mf g_{\leq-1}=\oplus_{\alpha>0}\mb Fe_{-\alpha}$.
We may normalize the root vectors by letting $(e_\alpha|e_{-\alpha})=1$.
Then, we have
$$
D(a)
=
-\sum_{\alpha>0}[e_{-\alpha},[e_{\alpha},a]]
\,,
$$
for every $a\in\mf g$.
In particular, for $h\in\mf g_0$, we have
$$
D(h)
=
-\sum_{\alpha>0}\alpha(h)h_\alpha
\,.
$$
This shows that, even though $D$ preserves the root space decomposition of $\mf g$
(hence all the root vectors $e_\alpha$ are its eigenvectors),
it does not act as a scalar on $\mf g_0$.
\end{example}

\subsection{The operator $L(z)$}
\label{sec:4.3}

We introduce some important $\End V$-valued polynomials in $z$,
and Laurent series in $z^{-1}$, with coefficients in $U(\mf g)$.
The first one is (cf. \eqref{20170623:eq4})
\begin{equation}\label{eq:A}
A(z)=z\id_V+U=z\id_V+\sum_{i\in I}u_i U^i
\,\,
\in U(\mf g)[z]\otimes\End(V)
\,.
\end{equation}
Here and further, we drop the tensor product sign when writing an element of 
$\mc U\otimes\End V$ ($\mc U$ being, in this case, the associative algebra $U(\mf g)[z]$).
Another important operator is
(keeping the same notation as in \cite{DSKV17})
\begin{equation}\label{eq:Arho}
A^{\rho}(z)
=
z\id_V+F+\pi_{\leq\frac12}U
=
z\id_V+F+\sum_{i\in I_{\leq\frac12}}u_i U^i
\,\,\in U(\mf g)[z]\otimes\End V
\,.
\end{equation}
There is a close connection between the operators $A(z)$ and $A^{\rho}(z)$,
which can be described in terms of the $U(\mf g)$-module $M$
defined in \eqref{eq:M}.
%
%
We extend, in the obvious way,
the left action of $U(\mf g)$ on the module $M$ \eqref{eq:M}
to a left action of the associative algebra 
$U(\mf g)[z]\otimes\End V$ 
on the module $M[z]\otimes\End V$.
(Later we shall also consider the further extension to a left action
of $U(\mf g)((z^{-1}))\otimes\End V$ 
on the module $M((z^{-1}))\otimes\End V$).
Applying \eqref{eq:A} and \eqref{eq:Arho} to $\bar1\,(=\bar 1\id_V)$, 
by \eqref{0225:eq4}, \eqref{eq:M}
we get the following identity
$$
A(z)\,\bar 1
=
A^\rho(z)\,\bar1
\,\in M[z]\otimes\End V
\,.
$$

Now we introduce 
the operator $L(z)$.
Consider the generalized quasideterminant (cf. \eqref{eq:linalg7})
\begin{equation}\label{eq:tildeL}
\widetilde L(z)
=
|A^\rho(z)+D|_{\Psi_{\frac d2},\Pi_{-\frac d2}}
=
\Big(\Pi_{-\frac d2}\big(z\id_V+F+\pi_{\leq\frac12}U+D\big)^{-1}\Psi_{\frac d2}\Big)^{-1}
\,,
\end{equation}
where $\Pi_{-\frac d2}$ and $\Psi_{\frac d2}$ are defined in \eqref{eq:chi}
and $D$ is the ``shift matrix'' \eqref{eq:D}.
We shall prove that,
if we view $A^\rho(z)+D$
as an element of the algebra $U(\mf g)((z^{-1}))\otimes\End V$,
then its generalized quasideterminant \eqref{eq:tildeL} exists,
and it lies in $U(\mf g)((z^{-1}))\otimes\Hom\big(V\big[-\frac d2\big],V\big[\frac d2\big]\big)$.
The matrix $L(z)$ is defined as the image in the module 
$M((z^{-1}))\otimes\Hom\big(V\big[-\frac d2\big],V\big[\frac d2\big]\big)$
of this quasideterminant:
\begin{equation}\label{eq:L}
L(z)
=
L_{\mf g,f,V}(z)
:=
\widetilde L(z)\bar 1
\,.
\end{equation}

The main result of the present Section is that
the entries of the coefficients of $L(z)$ actually lie in the $W$-algebra 
$W(\mf g,f)\subset M$.
This is stated in Theorem \ref{thm:main1} below.
Before stating it, we prove, 
in Section \ref{sec:4.4},
that the generalized quasideterminant defining $\widetilde L(z)$ exists.

\subsection{$\widetilde L(z)$ exists}
\label{sec:4.4}

\begin{proposition}\label{thm:L1}
\begin{enumerate}[(a)]
\item
The operator $A^\rho(z)+D$ is invertible in
$U(\mf g)((z^{-1}))\otimes\End V$.
\item
The operator $\Pi_{-\frac d2}(A^\rho(z)+D)^{-1}\Psi_{\frac d2}
\,\in U(\mf g)((z^{-1}))\otimes\Hom\big(V\big[\frac d2\big],V\big[-\frac d2\big]\big)$
is a Laurent series in $z^{-1}$
of degree (=the largest power of $z$ with non-zero coefficient) 
equal to $-d-1$,
and with leading coefficient $(-1)^{d}\Pi_{-\frac d2}F^d\Psi_{\frac d2}$.
In particular, it is invertible,
with inverse in
$U(\mf g)((z^{-1}))\otimes\Hom\big(V\big[-\frac d2\big],V\big[\frac d2\big]\big)$.
\item
Consequently,
the quasideterminant defining $\widetilde L(z)$ (cf. \eqref{eq:tildeL}) exists
and lies in $U(\mf g)((z^{-1}))\otimes\Hom\big(V\big[-\frac d2\big],V\big[\frac d2\big]\big)$.
\end{enumerate}
\end{proposition}
\begin{proof}
The operator 
$A^\rho(z)+D=z\id_V+F+\pi_{\leq\frac12}U+D$
is a polynomial in $z$ of degree $1$,
with leading coefficient $\id_V$.
Hence it is invertible in the algebra 
$U(\mf g)((z^{-1}))\otimes\End V$,
and its inverse can be computed by geometric series expansion:
\begin{equation}\label{eq:thm1-pr1}
(A^\rho(z)+D)^{-1}
=
\sum_{\ell=0}^\infty (-1)^\ell z^{-\ell-1}
\big(F+\pi_{\leq\frac12}U+D\big)^\ell
\,.
\end{equation}
This proves part (a). 
Next, we prove part (b).
We have, by \eqref{eq:thm1-pr1},
\begin{equation}\label{eq:thm1-pr2}
\Pi_{-\frac d2}(A^\rho(z)+D)^{-1}\Psi_{\frac d2}
=
\sum_{\ell=0}^\infty (-1)^\ell z^{-\ell-1}
\Pi_{-\frac d2}\big(F+\pi_{\leq\frac12}U+D\big)^\ell\Psi_{\frac d2}
\,.
\end{equation}
Note that
$F\in(\End V)[-1]$,
$U^i\in(\End V)[\geq-\frac12]$ for every $i\in I_{\leq\frac12}$,
and $D\in(\End V)[0]$.
Since $\im\Psi_{\frac d2}= V\big[\frac d2\big]$
and $\ker\Pi_{-\frac d2}=V\big[>-\frac d2\big]$,
we have
$$
\Pi_{-\frac d2}\big(F+\pi_{\leq\frac12}U+D\big)^\ell\Psi_{\frac d2}
=0
\,\,\text{ for every }\,\,
0\leq \ell<d
\,,
$$
and 
$$
\Pi_{-\frac d2}\big(F+\pi_{\leq\frac12}U+D\big)^d\Psi_{\frac d2}
=
\Pi_{-\frac d2}F^d\Psi_{\frac d2}
\,.
$$
Hence, by \eqref{eq:thm1-pr2},
$\Pi_{-\frac d2}(A^\rho(z)+D)^{-1}\Psi_{\frac d2}=(-1)^dz^{-d-1}\Pi_{-\frac d2}F^d\Psi_{\frac d2}+$
lower powers of $z$.
On the other hand, by representation theory of $\mf{sl}_2$,
the map
$\Pi_{-\frac d2}F^d\Psi_{\frac d2}:\,V\big[\frac d2\big]\to V\big[-\frac d2\big]$
is invertible.
Claim (b) follows.
Claim (c) is an obvious consequence of (a) and (b).
\end{proof}

\subsection{The first Main Theorem}
\label{sec:main-thm}

\begin{theorem}\label{thm:main1}
The entries of the coefficients of the operator $L(z)$ defined in \eqref{eq:L} 
lie in the $W$-algebra $W(\mf g,f)$:
$$
L(z):=|z\id_V+F+\pi_{\leq\frac12}U+D|_{\Psi_{\frac d2},\Pi_{-\frac d2}}\bar 1
\,\in
W(\mf g,f)((z^{-1}))\otimes\Hom\big(V\big[-\frac d2\big],V\big[\frac d2\big]\big)
\,.
$$
\end{theorem}

\section{Proof of Theorem \ref{thm:main1}}

We shall prove Theorem \ref{thm:main1}
in Section \ref{sec:4.6}.
Its proof will rely on the Main Lemma \ref{lem:main},
which will be stated in Section \ref{sec:4.5}
and proved in Sections \ref{sec:step1}-\ref{step5}.
In order to state (and prove) Lemma \ref{lem:main}, though,
we need to extend
the action of $U(\mf g)$ on the module $M=U(\mf g)/J$
to an action of the (completed) Rees algebra $\mc RU(\mf g)$
(and its extension $\mc R_\infty U(\mf g)$)
on the corresponding (completed) Rees module $\mc RM$.
This is the content of the next Section \ref{sec:ore}.

\subsection{Preliminaries: the Kazhdan filtration of $U(\mf g)$,
the Rees algebra $\mc RU(\mf g)$\!,
its localization $\mc R_\infty U(\mf g)$,
and the Rees module $\mc RM$
}
\label{sec:ore}

In the present section we review,
following  \cite[Sec.5]{DSKV17},
the construction of the (completed) Rees algebra $\mc RU(\mf g)$,
its extension $\mc R_\infty U(\mf g)$,
and the (completed) Rees module $\mc RM$.

First recall that,
associated to the grading \eqref{eq:grading} of $\mf g$,
we have the \emph{Kazhdan filtration} of $U(\mf g)$,
\begin{equation}\label{0312:eq1}
F_n U(\mf g)
=
\sum_{s-j_1-\dots-j_s\leq n}
\mf g_{j_1}\dots\mf g_{j_s}
\,\,,\,\,\,\,
n\in\frac12\mb Z
\,.
\end{equation}
In other words, $\{F_n U(\mf g)\}_{n\in\frac12\mb Z}$ 
is the increasing filtration of $U(\mf g)$ defined letting the degree,
called the \emph{conformal weight}, of $\mf g_j$ equal to $1-j$.
It has the following properties:
$F_{\Delta_1}U(\mf g)\cdot F_{\Delta_2}U(\mf g)\subset F_{\Delta_1+\Delta_2}U(\mf g)$,
and 
$[F_{\Delta_1}U(\mf g),F_{\Delta_2}U(\mf g)]\subset F_{\Delta_1+\Delta_2-1}U(\mf g)$.
Since $m-(f|m)$ is homogeneous in conformal weight,
the Kazhdan filtration induces the increasing filtration of the left ideal $J\subset U(\mf g)$,
given by
\begin{equation}\label{0312:eq2}
F_\Delta J
:=
\big(F_\Delta U(\mf g)\big)\cap J
=
\sum_{j\geq1}
(F_{\Delta+j-1}U(\mf g))
\big\{m-(f|m)\,\big|\,m\in\mf g_j\big\}
\,.
\end{equation}
Hence, we get the induced filtration of the quotient module $M=U(\mf g)/J$,
$$
F_\Delta M=F_\Delta U(\mf g)/F_\Delta J
\,\,,\,\,\,\,
\Delta\in\frac12\mb Z
\,.
$$
\begin{proposition}[{\cite[Sec.5.1]{DSKV17}}]\label{prop:rees1}
\begin{enumerate}[(a)]
\item
$F_n U(\mf g)=\delta_{n,0}\mb F\oplus F_nJ$ for every $n\leq 0$.
\item
$F_n M=\delta_{n,0}\mb F\bar 1$ for every $n\leq 0$.
\item
$F_0J\subset F_0U(\mf g)$ is a twosided ideal of codimension $1$,
and the corresponding quotient map 
is the algebra homomorphism
$\epsilon_0:\,F_0U(\mf g)\to\mb F$
given by the following formula:
\begin{equation}\label{0410:eq10}
\epsilon_0\big(\sum a_1\dots a_\ell\big) = \sum (f|a_1)\dots (f|a_\ell) 
\,.
\end{equation}
\item
The action of $F_0U(\mf g)$ on $F_0M=\mb F\bar 1$ is 
induced by the map \eqref{0410:eq10},
i.e. $u\bar 1=\epsilon(u)\bar 1$ for every $u\in F_0U(\mf g)$.
\end{enumerate}
\end{proposition}

The (completed) Rees algebra $\mc RU(\mf g)$
associated to the Kazhdan filtration 
is defined as follows
\begin{equation}\label{eq:reesb}
\mc RU(\mf g)
=
\widehat{\sum}_{n\in\frac12\mb Z}
z^{-n}F_{n}U(\mf g)
\,\subset U(\mf g)((z^{-\frac12}))
\,,
\end{equation}
where the completion is defined by allowing 
series with infinitely many negative integer powers of $z^{\frac12}$.
Note that 
$F_0U(\mf g)\subset\mc RU(\mf g)$ is a subalgebra of the Rees algebra
(but, for $n>0$, $F_nU(\mf g)$ is not contained in $\mc RU(\mf g)$).
Note also that $z^{-\frac12}\in\mc RU(\mf g)$
is a central element of the Rees algebra
(but, for $n>0$,  $z^n$ does not lie in $\mc RU(\mf g)$).
We can consider the left ideal $\mc RJ$ of the Rees algebra $\mc RU(\mf g)$, defined,
with the same notation as in \eqref{eq:reesb}, as
\begin{equation}\label{eq:reesI}
\mc RJ
=
\widehat{\sum}_{n\in\frac12\mb Z}
z^{-n}F_{n}J
\,\subset J((z^{-\frac12}))
\,.
\end{equation}
Taking the quotient of the Rees algebra $\mc RU(\mf g)$ by its left ideal $\mc RJ$ 
we get the corresponding Rees module
\begin{equation}\label{eq:reesM}
\mc RM
=
\mc RU(\mf g)/\mc RJ
=
\mb F\bar1\oplus\mc R_-M
\,,
\end{equation}
where
\begin{equation}\label{0411:eq1}
\mc R_-M
=
\widehat{\sum}_{n\geq\frac12}
z^{-n}F_{n}M
\,\subset z^{-\frac12}M[[z^{-\frac12}]]
\end{equation}
is a submodule of codimension $1$.
Obviously, $\mc RM$ is a cyclic module over $\mc RU(\mf g)$
generated by the cyclic element $\bar 1$.
\begin{proposition}[{\cite[Sec.5.2-3]{DSKV17}}]\label{prop:rees2}
\begin{enumerate}[(a)]
\item
The map 
$\epsilon_0:\,F_0U(\mf g)\to\mb F$
defined by \eqref{0410:eq10}
extends to an algebra homomorphism
$\epsilon:\,\mc RU(\mf g)\to\mb F$
given by
\begin{equation}\label{0406:eq1}
\epsilon\big(\sum_na_nz^n\big)=\epsilon_0(a_0)
\,.
\end{equation}
\item
$\mc RJ\cdot\mc RM\subset z^{-1}\mc RM$,
and $z^{-\frac12}\mc RM\subset\mc R_-M$.
\item
The action of $\mc RU(\mf g)$ 
on the quotient module $\mc RM/\mc R_-M=\mb F\bar1$
is induced by the map \eqref{0406:eq1},
i.e. $a(z)\bar 1\equiv \epsilon(a(z))\bar 1\mod\mc R_-M$
for every $a(z)\in\mc RU(\mf g)$.
\item
An element $a(z)\in\mc RU(\mf g)$ acts as an invertible endomorphism of $\mc RM$
if and only if $\epsilon(a(z))\neq0$.
\end{enumerate}
\end{proposition}

By Proposition \ref{prop:rees2}(d), 
an element $g(z)\in\mc RU(\mf g)$, with $\epsilon(g(z))\neq0$,
acts as an invertible endomorphism of the Rees module $\mc RM$.
But, in general, the inverse of $g(z)$ does not necessarily exist in the Rees algebra,
since its inverse may involve infinitely many positive powers of $z$.
We therefore localize the Rees algebra $\mc R U(\mf g)$ 
to its extension $\mc R_\infty U(\mf g)$
with the property that
all elements $g(z)\in\mc RU(\mf g)$ such that $\epsilon(g(z))\neq0$
are invertible in $\mc R_\infty U(\mf g)$.
This is stated in the following:
\begin{proposition}[{\cite[Sec.5.4]{DSKV17}}]\label{prop:rees3}
There exists an algebra extension $\mc R_\infty U(\mf g)$
of the Rees algebra $\mc R U(\mf g)$,
satisfying the following properties:
\begin{enumerate}[(a)]
\item
The map \eqref{0406:eq1} extends to an algebra homomorphism
$\epsilon:\,\mc R_\infty U(\mf g)\to\mb F$.
\item
The the left action of $\mc RU(\mf g)$ on the Rees module $\mc RM$
extends to a left action of $\mc R_\infty U(\mf g)$ on $\mc RM$,
and $\mc R_-M$ is preserved by this action.
\item
The action of $\mc R_\infty U(\mf g)$ on the quotient module $\mc RM/\mc R_-M=\mb F\bar 1$
is induced by the map $\epsilon$ in (a), i.e. 
$\alpha(z)\bar1\equiv\epsilon(\alpha(z))\bar1\,\mod\mc R_-M$
for every $\alpha(z)\in\mc R_\infty U(\mf g)$.
\item
For every $\alpha(z)\in\mc R_\infty U(\mf g)$ and every integer $N\geq0$,
there exist $\alpha_N(z)\in\mc RU(\mf g)$ such that
$(\alpha(z)-\alpha_N(z))\cdot\mc RM\subset z^{-N-1}\mc RM$;
\item
For an element $\alpha(z)\in\mc R_\infty U(\mf g)$,
the following conditions are equivalent:
\begin{enumerate}[(i)]
\item
$\alpha(z)$ is invertible in $\mc R_\infty U(\mf g)$;
\item
$\alpha(z)$ acts as an invertible endomorphism of $\mc RM$;
\item
$\epsilon(\alpha(z))\neq0$.
\end{enumerate}
\item
An operator $A(z)\in\mc R_\infty U(\mf g)\otimes\Hom(V_1,V_2)$,
where $V_1,V_2$ are vector spaces, 
is invertible if and only if $\epsilon(A(z))\in\Hom(V_1,V_2)$ is invertible.
\end{enumerate}
\end{proposition}

We shall also need to consider the Rees algebra 
of the $W$-algebra $W(\mf g,f)\subset M$,
induced from $M$:
\begin{equation}\label{eq:reesW}
\mc RW(\mf g,f)
=
\widehat{\sum}_{n\geq0}
z^{-n}F_{n}W(\mf g,f)
\,\subset \mc RM
\,,
\end{equation}
which is a subalgebra of the algebra $W(\mf g,f)[[z^{-\frac12}]]$.
\begin{proposition}[{\cite[Prop.5.14]{DSKV17}}]
\label{0413:prop1}
Let $\alpha(z)\in\mc R_\infty U(\mf g)$, $g(z)\in\mc RU(\mf g)$ and $w(z)\in\mc RM$
be such that
$\alpha(z)\bar1=g(z)\bar1=w(z)$.
Then, the following conditions are equivalent:
\begin{enumerate}[(i)]
\item
$[a,\alpha(z)]\bar1=0$ for all $a\in\mf g_{\geq\frac12}$;
\item
$[a,g(z)]\bar1=0$ for all $a\in\mf g_{\geq\frac12}$;
\item
$w(z)\in\mc RW(\mf g,f)$.
\end{enumerate}
\end{proposition}

\subsection{Statement of the Main Lemma}
\label{sec:4.5}

Recalling the definition \eqref{eq:reesb} of the Rees algebra,
we define the operators
\begin{equation}\label{20170623:eq10}
z^{-\Delta}U
:=
\sum_{i\in I}z^{\delta(i)-1}u_iU^i
\,,\,\,
\pi_{\leq\frac12}z^{-\Delta}U
:=
\sum_{\delta(i)\leq\frac 12}z^{\delta(i)-1}u_iU^i
\,\in\mc RU(\mf g)\otimes\End V
\,,
\end{equation}
where $\delta(i)$ is as in \eqref{20170623:eq6}, 
and for the second sum we are using the notation \eqref{20170623:eq7}.
The Main Lemma, on which the proof of Theorem \ref{thm:main1} is based,
is the following:
\begin{lemma}\label{lem:main}
The generalized quasideterminant $|\id_V+z^{-\Delta}U|_{\Psi_{\frac d2},\Pi_{-\frac d2}}$
exists in the space 
$\mc  R_\infty U(\mf g)\otimes\Hom\big(V\big[-\frac d2\big],V\big[\frac d2\big]\big)$,
and the following identity holds
in $\mc  RM\otimes\Hom\big(V\big[-\frac d2\big],V\big[\frac d2\big]\big)$:
\begin{equation}\label{0229:eq3}
|\id_V+z^{-\Delta}U|_{\Psi_{\frac d2},\Pi_{-\frac d2}}\!\bar 1
=
z^{-d-1}|z\id_V+F+\pi_{\leq\frac12}U+D|_{\Psi_{\frac d2},\Pi_{-\frac d2}}\!\bar 1
\,.
\end{equation}
\end{lemma}

\subsection{Step 1: Existence of the quasideterminant 
$|\id_V+z^{-\Delta}U|_{\Psi_{\frac d2},\Pi_{-\frac d2}}$}\label{sec:step1}

Consider the following semisimple endomorphism
\begin{equation}\label{eq:X}
z^X
=
\sum_{k\in\frac12\mb Z}z^k\id_{V[k]}
\,\in(\End V)[z^{\pm\frac12}]
\,,
\end{equation}
where $\id_{V[k]}:=\Psi_k\Pi_k\in\End V$ is the projection onto $V[k]\subset V$.
It is clearly an invertible element of the algebra
$\mb F[z^{\pm\frac12}]\otimes\End V
\subset
U(\mf g)((z^{-\frac12}))\otimes\End V$.
Its action on $V$ is given by
$$
z^X(v)=z^k v\,\text{ for }\, v\in V[k]
\,,
$$
and its adjoint action on $\End V$ is 
\begin{equation}\label{20170623:eq9}
z^{-X}Az^X=z^{-k}A
\,\text{ for }\,
A\in(\End V)[k]
\,.
\end{equation}
\begin{lemma}\label{lem:X}
The following identity holds
(in the algebra $U(\mf g)((z^{-\frac12}))\otimes\End V$):
\begin{equation}\label{0410:eq5}
\id_V+z^{-\Delta}U
=
z^{-1-X}(z\id_V+U)z^X
\,.
\end{equation}
\end{lemma}
\begin{proof}
We have
\begin{equation}\label{20170427:eq1}
z^{-1-X}Uz^X
=
z^{-1}\sum_{i\in I}
u_i z^{-X}U^iz^X
=
\sum_{i\in I}z^{\delta(i)-1}
u_i U^i
=z^{-\Delta}U
\,,
\end{equation}
where we used \eqref{20170623:eq9} for the second equality,
and \eqref{20170623:eq10} for the last equality.
\end{proof}
\begin{lemma}\label{0410:lem3}
The operator $\id_V+z^{-\Delta}U$ is invertible 
in $\mc RU(\mf g)\otimes\End V$.
\end{lemma}
\begin{proof}
Clearly, $z\id_V+U$ is invertible,
by geometric series expansion, 
in the algebra $U(\mf g)((z^{-\frac12}))\otimes\End V$.
It follows by \eqref{0410:eq5} that $\id_V+z^{-\Delta}U$
is invertible in $U(\mf g)((z^{-\frac12}))\otimes\End V$
as well.
We need to prove that the coefficients of the inverse operator 
$(\id_V+z^{-\Delta}U)^{-1}$ actually lie in the Rees algebra $\mc RU(\mf g)$.
The inverse operator $(\id_V+z^{-\Delta}U)^{-1}$
is easily computed by \eqref{0410:eq5} and geometric series expansion:
\begin{equation}\label{0410:eq6}
\begin{array}{l}
\displaystyle{
\vphantom{\Big(}
(\id_V+z^{-\Delta}U)^{-1}
=
z^{1-X}(z\id_V+U)^{-1}z^X
=
\sum_{\ell=0}^\infty(-1)^\ell z^{-\ell}
z^{-X}U^\ell z^X
} \\
\displaystyle{
\vphantom{\Big(}
=
\sum_{\ell=0}^\infty (-1)^\ell z^{-\ell}
\sum_{i_1,\dots,i_\ell\in I}
z^{\delta(i_1)+\dots+\delta(i_\ell)}
u_{i_1}\dots u_{i_\ell}
U^{i_1}\dots U^{i_\ell}
\,,}
\end{array}
\end{equation}
where we used, for the last equality,
the notation \eqref{20170623:eq6} and equation \eqref{20170623:eq9}.
The monomial $u_{i_1}\dots u_{i_\ell}\in U(\mf g)$
has conformal weight
$$
\Delta=\ell-\delta(i_1)-\dots-\delta(i_\ell)
\,.
$$
Hence, 
the claim follows by \eqref{0410:eq6} and the definition 
of the Rees algebra $\mc RU(\mf g)$.
\end{proof}
\begin{lemma}\label{0410:lem4}
Applying the homomorphism $\epsilon:\,\mc RU(\mf g)\to\mb F$ defined in \eqref{0406:eq1}
to the entries of the operator 
$\Pi_{-\frac d2}(\id_V+z^{-\Delta}U)^{-1}\Psi_{\frac d2}
\in\mc RU(\mf g)\otimes\Hom\big(V\big[\frac d2\big],V\big[-\frac d2\big]\big)$,
we get
\begin{equation}\label{0410:eq7}
\epsilon\big(\Pi_{-\frac d2}(\id_V+z^{-\Delta}U)^{-1}\Psi_{\frac d2}\big)
=
(-1)^{d}\Pi_{-\frac d2}F^d\Psi_{\frac d2}
\,\in\Hom\big(V\big[\frac d2\big],V\big[-\frac d2\big]\big)
\,.
\end{equation}
\end{lemma}
\begin{proof}
Note that
$\Pi_{-\frac d2}U^{i_1}\dots U^{i_\ell}\Psi_{\frac d2}$
is zero unless $\delta(i_1)+\dots+\delta(i_\ell)=d$.
Equation \eqref{0410:eq6} then gives
\begin{equation}\label{0410:eq6c}
\Pi_{-\frac d2}(\id_V+z^{-\Delta}U)^{-1}\Psi_{\frac d2}
=
\sum_{\ell=0}^\infty (-1)^\ell z^{d-\ell}
\!\!\!
\sum_{i_1,\dots,i_\ell\in I}
u_{i_1}\dots u_{i_\ell}
\,\Pi_{-\frac d2}
U^{i_1}\dots U^{i_\ell}
\Psi_{\frac d2}
\,.
\end{equation}
Applying the map $\epsilon$ to both sides we get,
by \eqref{0406:eq1} and \eqref{0410:eq10},
$$
\begin{array}{l}
\displaystyle{
\vphantom{\Big(}
\epsilon\big(\Pi_{-\frac d2}(\id_V+z^{-\Delta}U)^{-1}\Psi_{\frac d2}\big)
=
(-1)^d 
\sum_{i_1,\dots,i_d\in I}
(f|u_{i_1})\dots (f|u_{i_d})
\Pi_{-\frac d2}
U^{i_1}\dots U^{i_d}
\Psi_{\frac d2}
} \\
\displaystyle{
\vphantom{\Big(}
=
(-1)^d\Pi_{-\frac d2}F^d\Psi_{\frac d2}
\,,}
\end{array}
$$
as claimed.
\end{proof}
\begin{lemma}\label{0410:lem5}
The operator $\Pi_{-\frac d2}(\id_V+z^{-\Delta}U)^{-1}\Psi_{\frac d2}\in\mc RU(\mf g)
\otimes\Hom\big(V\big[\frac d2\big],V\big[-\frac d2\big]\big)$ 
has an inverse in 
$\mc R_\infty U(\mf g)\otimes\Hom\big(V\big[-\frac d2\big],V\big[\frac d2\big]\big)$.
\end{lemma}
\begin{proof}
It follows by Proposition \ref{prop:rees3}(f) 
and Lemma \ref{0410:lem4},
since, obviously, 
$\Pi_{-\frac d2}F^d\Psi_{\frac d2}:\,V\big[\frac d2\big]\to V\big[-\frac d2\big]$
is bijective.
\end{proof}
\begin{proposition}\label{0410:prop4}
The quasideterminant 
$|\id_V+z^{-\Delta}U|_{\Psi_{\frac d2},\Pi_{-\frac d2}}$ 
exists and lies in
$\mc R_\infty U(\mf g)\otimes\Hom\big(V\big[-\frac d2\big],V\big[\frac d2\big]\big)$.
\end{proposition}
\begin{proof}
It is an immediate consequence of Lemmas \ref{0410:lem3} and \ref{0410:lem5}.
\end{proof}

\subsection{Step 2: Preliminary computations}\label{step3}

Recall the definition \eqref{eq:X} of the matrix $z^X$
and its adjoint action \eqref{20170623:eq9}.
We have
\begin{equation}\label{0413:eq6}
\begin{array}{l}
\displaystyle{
\vphantom{\Big(}
z^{-X}Uz^X=z^{1-\Delta}U
\,\,,\,\,\,\,
z^{-X}\pi_{\leq\frac12}Uz^X=\pi_{\leq\frac12}z^{1-\Delta}U
\,\,,\,\,\,\,
z^{-X}Fz^X=zF
\,,} \\
\displaystyle{
\vphantom{\Big(}
z^{-X}\id_Vz^X=\id_V
\,\,,\,\,\,\,
z^{-X}Dz^X=D
\,,}
\end{array}
\end{equation}
from which we get the following identity:
\begin{equation}\label{0412:eq8}
z\id_V+F+\pi_{\leq\frac12}U+D
=
z^{1+X}(\id_V+F+\pi_{\leq\frac12}z^{-\Delta}U+z^{-1}D)z^{-X}
\,.
\end{equation}
Taking the  $(\Psi_{\frac d2},\Pi_{-\frac d2})$-quasideterminant of both sides of \eqref{0412:eq8}
we get, by \eqref{eq:linalg7},
\begin{equation}\label{0412:eq8b}
\begin{array}{l}
\displaystyle{
\vphantom{\Big(}
|z\id_V+F+\pi_{\leq\frac12}U+D|_{\Psi_{\frac d2},\Pi_{-\frac d2}}
} \\
\displaystyle{
\vphantom{\Big(}
=
\big(
\Pi_{-\frac d2}
z^{X}(\id_V+F+\pi_{\leq\frac12}z^{-\Delta}U+z^{-1}D)^{-1}z^{-1-X}
\Psi_{\frac d2}
\big)^{-1}
} \\
\displaystyle{
\vphantom{\Big(}
=
z^{1+d}
\big(
\Pi_{-\frac d2}
(\id_V+F+\pi_{\leq\frac12}z^{-\Delta}U+z^{-1}D)^{-1}
\Psi_{\frac d2}
\big)^{-1}
} \\
\displaystyle{
\vphantom{\Big(}
=
z^{1+d}
|\id_V+F+\pi_{\leq\frac12}z^{-\Delta}U+z^{-1}D|_{\Psi_{\frac d2},\Pi_{-\frac d2}}
\,,}
\end{array}
\end{equation}
since $\Pi_{-\frac d2}z^X=z^{-\frac d2}\Pi_{-\frac d2}$ and $z^{-X}\Psi_{\frac d2}=z^{-\frac d2}\Psi_{\frac d2}$.
In view of \eqref{0412:eq8b},
equation \eqref{0229:eq3} becomes
\begin{equation}\label{0229:eq3b}
|\id_V+z^{-\Delta}U|_{\Psi_{\frac d2},\Pi_{-\frac d2}}\bar 1
=
|\id_V+F+\pi_{\leq\frac12}z^{-\Delta}U+z^{-1}D|_{\Psi_{\frac d2},\Pi_{-\frac d2}}\bar 1
\,,
\end{equation}
which we need to prove,
in order to complete the proof of Lemma \ref{lem:main}.

Let us compute the quasideterminants
in the LHS and the RHS of equation \eqref{0229:eq3b} 
applying Proposition \ref{prop:linalg},
i.e. using formula \eqref{eq:linalg6} for the quasideterminants,
with the short exact sequences $\chi_1,\chi_2$ in \eqref{eq:chi3}.
For the quasideterminant in the LHS, we have
\begin{equation}\label{0330:eq5}
\begin{array}{l}
\displaystyle{
\vphantom{\Big(}
|\id_V+z^{-\Delta}U|_{\Psi_{\frac d2},\Pi_{-\frac d2}}
=
\Psi_{\frac d2}^{-1}\Big(
\id_V+z^{-\Delta}U
}\\
\displaystyle{
\vphantom{\Big(}
-
(\id_V+z^{-\Delta}U)\Psi_{>-\frac d2}
\big(\Pi_{<\frac d2}(\id_V+z^{-\Delta}U)\Psi_{>-\frac d2}\big)^{-1}
\Pi_{<\frac d2}(\id_V+z^{-\Delta}U)
\Big)\Pi_{-\frac d2}^{-1}
\,.}
\end{array}
\end{equation}
We already know that the expression in the RHS is well defined,
i.e. the operator in parenthesis induces a well defined map
from $V\big[-\frac d2\big]$ to $V\big[\frac d2\big]$.
Hence, we can replace $\Psi_{\frac d2}^{-1}$ and $\Pi_{-\frac d2}^{-1}$ in the RHS
with $\Pi_{\frac d2}$ and $\Psi_{-\frac d2}$ respectively (cf. \eqref{eq:chi}).
Note also that 
$\Pi_{\frac d2}\id_V\Psi_{-\frac d2}=0$
and
$\Pi_{\frac d2}z^{-\Delta}U\Psi_{-\frac d2}=z^{-1-d}\Pi_{\frac d2}U\Psi_{-\frac d2}$.
Hence, we can rewrite the RHS of \eqref{0330:eq5} as
\begin{equation}\label{0330:eq5b}
\begin{array}{l}
\displaystyle{
\vphantom{\big(}
z^{-1-d}\Pi_{\frac d2}U\Psi_{-\frac d2}
-
\Pi_{\frac d2}
(\id_V+z^{-\Delta}U)\Psi_{>-\frac d2}
} \\
\displaystyle{
\vphantom{\big(}
\times
\big(\Pi_{<\frac d2}(\id_V+z^{-\Delta}U)\Psi_{>-\frac d2}\big)^{-1}
\Pi_{<\frac d2}(\id_V+z^{-\Delta}U)\Psi_{-\frac d2}
\,.}
\end{array}
\end{equation}
Similarly, we use formula \eqref{eq:linalg6}
to compute the quasideterminant in the RHS of \eqref{0229:eq3b}.
We have
\begin{equation}\label{0330:eq6}
\begin{array}{l}
\displaystyle{
\vphantom{\Big(}
|\id_V+F+\pi_{\leq\frac12}z^{-\Delta}U+z^{-1}D|_{\Psi_{\frac d2},\Pi_{-\frac d2}}
}\\
\displaystyle{
\vphantom{\Big(}
=
\Pi_{\frac d2}
(\id_V\!+\!F\!+\!\pi_{\leq\frac12}z^{-\Delta}U\!+\!z^{-1}D)
\Psi_{-\frac d2}
-
\Pi_{\frac d2}
(\id_V\!+\!F\!+\!\pi_{\leq\frac12}z^{-\Delta}\!U+z^{-1}\!D)
\Psi_{>-\frac d2}
}\\
\displaystyle{
\vphantom{\Big(}
\times \!
\big(
\Pi_{\!<\frac d2}
(\id_V\!+\!F\!+\!\pi_{\leq\frac12}z^{-\Delta}\!U\!+\!z^{-1}\!D)
\Psi_{\!>\!-\!\frac d2}
\big)^{-1}
\Pi_{\!\!<\frac d2}
(\id_V\!+\!F\!+\!\pi_{\!\leq\frac12}z^{-\Delta}U\!+\!z^{-1}\!D)
\Psi_{\!-\!\frac d2}
}\\
\displaystyle{
\vphantom{\Big(}
=
z^{-1-d}\Pi_{\frac d2}U\Psi_{-\frac d2}
-
\Pi_{\frac d2}
(\id_V+z^{-\Delta}U)
\Psi_{>-\frac d2}
}\\
\displaystyle{
\vphantom{\Big(}
\times\big(
\Pi_{<\frac d2}
(\id_V\!+\!F\!+\!\pi_{\leq\frac12}z^{-\Delta}U\!+\!z^{-1}D)
\Psi_{>-\frac d2}
\big)^{-1}
\Pi_{<\frac d2}
(\id_V\!+\!z^{-\Delta}U\!+\!z^{-1}D)
\Psi_{-\frac d2}
\,,}
\end{array}
\end{equation}
where we used, for the second equality,
equations \eqref{20170623:eq8} and the obvious identities
$\Pi_{\frac d2}\id_V\Psi_{-\frac d2}=0$, $\Pi_{\frac d2}F=F\Psi_{-\frac d2}=0$,
and 
$$
\begin{array}{l}
\displaystyle{
\vphantom{\Big(}
\Pi_{\frac d2}\pi_{\leq\frac12}z^{-\Delta}U\Psi_{-\frac d2}=z^{-1-d}\Pi_{\frac d2}U\Psi_{-\frac d2}
\,,} \\
\displaystyle{
\vphantom{\Big(}
\Pi_{\frac d2}\pi_{\leq\frac12}z^{-\Delta}U=\Pi_{\frac d2}z^{-\Delta}U
\,\,,\,\,\,\,
\pi_{\leq\frac12}z^{-\Delta}U\Psi_{-\frac d2}=z^{-\Delta}U\Psi_{-\frac d2}
\,.}
\end{array}
$$
In view of \eqref{0330:eq5b} and \eqref{0330:eq6},
in order to prove equation \eqref{0229:eq3b} 
it suffices to prove
the following equation
\begin{equation}\label{0330:eq7}
\begin{array}{l}
\displaystyle{
\vphantom{\Big(}
\big(\Pi_{<\frac d2}(\id_V+z^{-\Delta}U)\Psi_{>-\frac d2}\big)^{-1}
\Pi_{<\frac d2}(\id_V+z^{-\Delta}U)\Psi_{-\frac d2}
\bar 1
} \\
\displaystyle{
\vphantom{\Big(}
=
\big(
\Pi_{<\frac d2}
(\id_V\!+\!F\!+\!\pi_{\leq\frac12}z^{-\Delta}U\!+\!z^{-1}D)
\Psi_{>-\frac d2}
\big)^{-1}
\Pi_{<\frac d2}
(\id_V\!+\!z^{-\Delta}U\!+\!z^{-1}D)
\Psi_{-\frac d2}
\bar 1
\,,}
\end{array}
\end{equation}
which we are left to prove.

To simplify notation, we introduce the operators
$A,B\in\mc RU(\mf g)\otimes\Hom\big(V\big[>-\frac d2\big],V\big[<\frac d2\big]\big)$
and 
$v,w\in\mc RU(\mf g)\otimes\Hom\big(V\big[-\frac d2\big],V\big[<\frac d2\big]\big)$,
defined as follows
\begin{equation}\label{0330:eq8}
\begin{array}{l}
\vphantom{\Big(}
\displaystyle{
A:=
\Pi_{<\frac d2}
(\id_V+F+\pi_{\leq\frac12}z^{-\Delta}U+z^{-1}D)
\Psi_{>-\frac d2}
\,,} \\
\vphantom{\Big(}
\displaystyle{
B:=
\Pi_{<\frac d2}(\id_V+z^{-\Delta}U)\Psi_{>-\frac d2}-A
} \\
\vphantom{\Big(}
\displaystyle{
\,\,\,\,\,\,\,\,\,
=
\sum_{\delta(i)\geq1}
(z^{-\Delta}u_i-(f|u_i))\,
\Pi_{<\frac d2} U^i\Psi_{>-\frac d2}
-
z^{-1}\Pi_{<\frac d2} D\Psi_{>-\frac d2}
\,,} \\
\vphantom{\Big(}
\displaystyle{
v:=
\Pi_{<\frac d2}
(\id_V+z^{-\Delta}U+z^{-1}D)
\Psi_{-\frac d2}
\,,} \\
\vphantom{\Big(}
\displaystyle{
w:=
\Pi_{<\frac d2}(\id_V+z^{-\Delta}U)\Psi_{-\frac d2}-v
=
-z^{-1} \Pi_{<\frac d2} D\Psi_{-\frac d2}
\,,} 
\end{array}
\end{equation}
where we use notation (cf. \eqref{20170623:eq6}):
$$
z^{-\Delta}u_i=z^{\delta(i)-1}u_i
\,\,\text{ for }\,\,
i\in I
\,.
$$
Using notation \eqref{0330:eq8},
equation \eqref{0330:eq7} can be rewritten as follows
\begin{equation}\label{0330:eq9}
(A+B)^{-1}(v+w)\bar1
=
A^{-1}v\bar1
\,\in
\mc RM\otimes\Hom\big(V\big[-\frac d2\big],V\big[>-\frac d2\big]\big)
\,.
\end{equation}

\subsection{Step 3: the key computation}\label{step4}

For every $i\in I_{\geq1}$, denote
\begin{equation}\label{key-X}
X_i
=
(z^{-\Delta}u_i-(f|u_i))
A^{-1}v\bar1
\,\in\mc RM\otimes
\Hom\big(V\big[-\frac d2\big],V\big[>-\frac d2\big]\big)
\,.
\end{equation}
We also let $X_i=0$ for $i\in I_{\leq\frac12}$.
\begin{lemma}\label{lem:key}
For every $i\in I_{\geq1}$ we have, in notation \eqref{20170623:eq7}:
\begin{equation}\label{key-eq}
\begin{array}{l}
\displaystyle{
\vphantom{\Big(}
X_i
+
z^{-1}
\sum_{1\leq\delta(j)\leq\delta(i)+\frac12}
A^{-1}
\Pi_{<\frac d2}
[U^j ,U_i]
\Psi_{>-\frac d2}
X_j
} \\
\displaystyle{
\vphantom{\Big(}
=
-z^{-1}
\Pi_{>-\frac d2} U_i
\big(
\Psi_{>-\frac d2}
A^{-1}
v
-
\Psi_{-\frac d2}
\big)
\bar1
} \\
\displaystyle{
\vphantom{\Big(}
\,\,\,\,\,\,+
z^{-2}
A^{-1}
\Pi_{<\frac d2} 
[D,U_i]
\big(
\Psi_{>-\frac d2}
A^{-1}v
-\Psi_{-\frac d2}
\big)
\bar1
\,.}
\end{array}
\end{equation}
\end{lemma}
\begin{proof}
Recall that $(z^{-\Delta}u_i-(f|u_i))\bar 1=0$ in $\mc RM$ for every $i\in I_{\geq1}$.
Hence, 
\begin{equation}\label{eq:prX1}
\begin{array}{l}
\displaystyle{
\vphantom{\Big(}
X_i
=
-A^{-1}
[z^{-\Delta}u_i,A]
A^{-1}v\bar1
+A^{-1}
[z^{-\Delta}u_i,v]
\bar1
} \\
\displaystyle{
\vphantom{\Big(}
\,\,\,\,\,\,\,
=
-\sum_{j\in I_{\leq\frac12}}
A^{-1}
\big[
z^{-\Delta}u_i,
z^{-\Delta}u_j
\big]
\Pi_{<\frac d2}
U^j
\Psi_{>-\frac d2}
A^{-1}v\bar1
} \\
\displaystyle{
\vphantom{\Big(}
\,\,\,\,\,\,\,\,\,\,\,\,\,\,\,
+
\sum_{j\in I}
A^{-1}
\big[
z^{-\Delta}u_i,
z^{-\Delta}u_j
\big]
\Pi_{<\frac d2}
U^j
\Psi_{-\frac d2}
\bar1
\,.}
\end{array}
\end{equation}
By the definition of conformal weight,
we have 
$$
\big[
z^{-\Delta}u_i,
z^{-\Delta}u_j
\big]
=
z^{-1-\Delta}[u_i,u_j]
\,.
$$
Moreover, by the completeness relations,
we have the identities (using notation \eqref{20170623:eq7})
$$
\sum_{j\in I}[u_i,u_j]\, U^j
=
\sum_{j\in I}u_j [U^j,U_i]
\,\,\text{ and }\,\,
\sum_{\delta(j)\leq\frac12}[u_i,u_j]\, U^j
=
\sum_{\delta(j)\leq\delta(i)+\frac12}u_j [U^j,U_i]
\,.
$$
Hence, \eqref{eq:prX1} gives
\begin{equation}\label{eq:prX2}
\begin{array}{l}
\displaystyle{
\vphantom{\Big(}
X_i
=
-z^{-1}
\sum_{\delta(j)\leq\delta(i)+\frac12}
A^{-1}
(z^{-\Delta}u_j)
\Pi_{<\frac d2}
\big[U^j,U_i\big]
\Psi_{>-\frac d2}
A^{-1}v\bar1
} \\
\displaystyle{
\vphantom{\Big(}
\,\,\,\,\,\,\,\,\,\,\,\,\,\,\,
+
z^{-1}
\sum_{j\in I}
A^{-1}
(z^{-\Delta}u_j)
\Pi_{<\frac d2}
\big[U^j,U_i\big]
\Psi_{-\frac d2}
\bar1
\,.}
\end{array}
\end{equation}
Since, by assumption, $i\in I_{\geq1}$,
we have $\im U_i\subset V\big[>-\frac d2\big]$ and $V\big[\frac d2\big]\subset\ker U_i$.
As a consequence, we have the following identities
(cf. \eqref{eq:chi})
\begin{equation}\label{20170626:eq2}
U_i
=
\Psi_{>-\frac d2}\Pi_{>-\frac d2} U_i
\,\,\text{ and }\,\,
U_i
=
U_i\Psi_{<\frac d2}\Pi_{<\frac d2}
\,.
\end{equation}
We can therefore rewrite \eqref{eq:prX2} as follows
\begin{equation}\label{eq:prX3}
\begin{array}{l}
\displaystyle{
\vphantom{\Big(}
X_i
=
-z^{-1}
\sum_{\delta(j)\leq\delta(i)+\frac12}
A^{-1}
(z^{-\Delta}u_j)
\Pi_{<\frac d2}
U^j
\Psi_{>-\frac d2}
\Pi_{>-\frac d2} U_i\Psi_{>-\frac d2}
A^{-1}v\bar1
} \\
\displaystyle{
\vphantom{\Big(}
+z^{-1}
\sum_{\delta(j)\leq\delta(i)+\frac12}
A^{-1}
(z^{-\Delta}u_j)
\Pi_{<\frac d2} U_i\Psi_{<\frac d2}
\Pi_{<\frac d2}
U^j
\Psi_{>-\frac d2}
A^{-1}v\bar1
} \\
\displaystyle{
\vphantom{\Big(}
+
z^{-1}
\sum_{j\in I}
A^{-1}
(z^{-\Delta}u_j)
\Pi_{<\frac d2}
U^j
\Psi_{>-\frac d2}
\Pi_{>-\frac d2} U_i\Psi_{-\frac d2}
\bar1
} \\
\displaystyle{
\vphantom{\Big(}
-
z^{-1}
\sum_{j\in I}
A^{-1}
(z^{-\Delta}u_j)
\Pi_{<\frac d2} U_i\Psi_{<\frac d2}
\Pi_{<\frac d2}
U^j
\Psi_{-\frac d2}
\bar1
\,.}
\end{array}
\end{equation}
Recalling the definitions \eqref{0330:eq8} of $A$ and $v$, we have the following identities:
$$
\begin{array}{l}
\displaystyle{
\vphantom{\Big(}
\sum_{\delta(j)\leq\delta(i)+\frac12}
(z^{-\Delta}u_j)
\Pi_{<\frac d2}
U^j
\Psi_{>-\frac d2}
} \\
\displaystyle{
\vphantom{\Big(}
=
A
+
\sum_{1\leq \delta(j)\leq \delta(i)+\frac12}
(z^{-\Delta}u_j-(f|u_j))
\Pi_{<\frac d2}
U^j
\Psi_{>-\frac d2}
-
\Pi_{<\frac d2}
(\id_V+z^{-1}D)
\Psi_{>-\frac d2}
\,,} \\
\displaystyle{
\vphantom{\Big(}
\sum_{j\in I}
(z^{-\Delta}u_j)
\Pi_{<\frac d2}
U^j
\Psi_{>-\frac d2}
} \\
\displaystyle{
\vphantom{\Big(}
=
A
+
\sum_{\delta(j)\geq1}
(z^{-\Delta}u_j-(f|u_j))
\Pi_{<\frac d2}
U^j
\Psi_{>-\frac d2}
-
\Pi_{<\frac d2}
(\id_V+z^{-1}D)
\Psi_{>-\frac d2}
\,,} \\
\displaystyle{
\vphantom{\Big(}
\sum_{j\in I}
(z^{-\Delta}u_j)
\Pi_{<\frac d2}
U^j
\Psi_{-\frac d2}
=
v
-
\Pi_{<\frac d2}(\id_V+z^{-1}D)\Psi_{-\frac d2}
\,.}
\end{array}
$$
Hence, the first term in the RHS of \eqref{eq:prX3} 
can be rewritten as
\begin{equation}\label{eq:prX4}
\begin{array}{l}
\displaystyle{
\vphantom{\Big(}
-z^{-1}
\sum_{1\leq \delta(j)\leq \delta(i)+\frac12}
A^{-1}
\Pi_{<\frac d2}
U^j U_i\Psi_{>-\frac d2}
X_j
-z^{-1}
\Pi_{>-\frac d2} U_i\Psi_{>-\frac d2}
A^{-1}v\bar1
} \\
\displaystyle{
\vphantom{\Big(}
+z^{-1}
A^{-1}
\Pi_{<\frac d2}
(\id_V+z^{-1}D)
U_i\Psi_{>-\frac d2}
A^{-1}v\bar1
\,,}
\end{array}
\end{equation}
the second term in the RHS of \eqref{eq:prX3} becomes
\begin{equation}\label{eq:prX5}
\begin{array}{l}
\displaystyle{
\vphantom{\Big(}
+z^{-1}
\sum_{1\leq \delta(j)\leq \delta(i)+\frac12}
A^{-1}
\Pi_{<\frac d2} U_i
U^j
\Psi_{>-\frac d2}
X_j
+z^{-1}
A^{-1}
\Pi_{<\frac d2} U_i\Psi_{<\frac d2}
v\bar1
} \\
\displaystyle{
\vphantom{\Big(}
-z^{-1}
A^{-1}
\Pi_{<\frac d2} U_i
(\id_V+z^{-1}D)
\Psi_{>-\frac d2}
A^{-1}v\bar1
\,,}
\end{array}
\end{equation}
the third term in the RHS of \eqref{eq:prX3} becomes
\begin{equation}\label{eq:prX6}
\begin{array}{l}
\displaystyle{
\vphantom{\Big(}
z^{-1}
\Pi_{>-\frac d2} U_i\Psi_{-\frac d2}
\bar1
-
z^{-1}
A^{-1}
\Pi_{<\frac d2}
(\id_V+z^{-1}D)
U_i\Psi_{-\frac d2}
\bar1
\,,}
\end{array}
\end{equation}
and the last term in the RHS of \eqref{eq:prX3} becomes
\begin{equation}\label{eq:prX7}
\begin{array}{l}
\displaystyle{
\vphantom{\Big(}
-
z^{-1}
A^{-1}
\Pi_{<\frac d2} U_i\Psi_{<\frac d2}
v
\bar1
+z^{-1}
A^{-1}
\Pi_{<\frac d2} U_i
(\id_V+z^{-1}D)\Psi_{-\frac d2}
\bar1
\,.}
\end{array}
\end{equation}
Combining \eqref{eq:prX4}-\eqref{eq:prX7}, we get \eqref{key-eq}.
\end{proof}
\begin{lemma}\label{lem:key2}
The unique solution of equation \eqref{key-eq} is (for $i\in I_{\geq1}$):
\begin{equation}\label{key-sol}
X_i
=
-z^{-1}
\Pi_{>-\frac d2} U_i
\big(
\Psi_{>-\frac d2}
A^{-1}
v
-
\Psi_{-\frac d2}
\big)
\bar1
\,.
\end{equation}
\end{lemma}
\begin{proof}
First, we prove that \eqref{key-sol} solves equation \eqref{key-eq}.
Note that the first term in the LHS of \eqref{key-eq} equals, by \eqref{key-sol},
the first term in the RHS of \eqref{key-eq}.
We hence need to prove that
the second terms in the LHS and RHS of \eqref{key-eq} coincide:
\begin{equation}\label{20170626:eq1}
\begin{array}{l}
\displaystyle{
\vphantom{\Big(}
-z^{-2}
\sum_{1\leq\delta(j)\leq\delta(i)+\frac12}
A^{-1}
\Pi_{<\frac d2}
[U^j ,U_i]
\Psi_{>-\frac d2}
\Pi_{>-\frac d2} U_j
\big(
\Psi_{>-\frac d2}
A^{-1}
v
-
\Psi_{-\frac d2}
\big)
\bar1
} \\
\displaystyle{
\vphantom{\Big(}
=
z^{-2}
A^{-1}
\Pi_{<\frac d2} 
[D,U_i]
\big(
\Psi_{>-\frac d2}
A^{-1}v
-\Psi_{-\frac d2}
\big)
\bar1
\,.}
\end{array}
\end{equation}
Recalling the first equation of \eqref{20170626:eq2},
equation \eqref{20170626:eq1} is established once we prove the following identity:
\begin{equation}\label{20170626:eq3}
\sum_{1\leq\delta(j)\leq\delta(i)+\frac12}
[U^j ,U_i] U_j
=
-[D,U_i]
\,.
\end{equation}
By the definition \eqref{eq:D} of the shift matrix $D$ and the Leibniz rule, we have
\begin{equation}\label{20170626:eq4}
-[D,U_i]
=
\sum_{\delta(j)\geq1}
\big(
[U^j,U_i]U_j
+U^j[U_j,U_i]
\big)
\,.
\end{equation}
On the other hand, by the duality of the bases $\{U_j\},\,\{U^j\}$ and the invariance of the trace form,
we have
\begin{equation}\label{20170626:eq5}
\begin{array}{l}
\displaystyle{
\vphantom{\Big(}
\sum_{\delta(j)\geq1}
U^j[U_j,U_i]
=
\sum_{\delta(j)\geq1}\sum_{k\in I}
([U_j,U_i]|U^k)
U^jU_k
=
} \\
\displaystyle{
\vphantom{\Big(}
-\sum_{\delta(k)\geq\delta(i)+1}
\sum_{j\in I}
(U_j|[U^k,U_i])
U^jU_k
=
-\sum_{\delta(k)\geq\delta(i)+1}
[U^k,U_i]U_k
\,.}
\end{array}
\end{equation}
Combining \eqref{20170626:eq4} and \eqref{20170626:eq5},
we get equation \eqref{20170626:eq3}.

The uniqueness of the solution of equation \eqref{key-eq} is clear. 
Indeed, equation \eqref{key-eq} has the matrix form
$(\id+z^{-1}M)X=Y$,
where $X$ is the column vector $(X_i)_{\delta(i)\geq1}$, 
with entries in the vector space $V=\mc RM\otimes\Hom\big(V\big[-\frac d2\big],V\big[>-\frac d2\big]\big)$,
$Y$ is the analogous column vector defined by the RHS of \eqref{key-eq},
and $M$ is some matrix with entries 
in $\mc RU(\mf g)\otimes\Hom\big(V\big[>-\frac d2\big],V\big[>-\frac d2\big]\big)$,
which is an algebra acting on the vector space $V$.
But then the matrix $\id+z^{-1}M$ can be inverted by geometric series expansion.
\end{proof}
\begin{corollary}\label{lem:key3}
We have (recall notation \eqref{0330:eq8})
\begin{equation}\label{key-cor}
B A^{-1}v\bar1=w\bar1
\,.
\end{equation}
\end{corollary}
\begin{proof}
By the definitions \eqref{0330:eq8} of $B$, the definition \eqref{key-X} of $X_i$
and its formula \eqref{key-sol}, we have
\begin{equation}\label{0331:eq2}
\begin{array}{l}
\displaystyle{
\vphantom{\Big(}
B A^{-1}v\bar1
=
\sum_{\delta(i)\geq1}
\Pi_{<\frac d2} U^i\Psi_{>-\frac d2}
X_i
-
z^{-1}\Pi_{<\frac d2} D\Psi_{>-\frac d2}
A^{-1}v\bar1
} \\
\displaystyle{
\vphantom{\Big(}
=
\!-\!z^{\!-\!1}
\!\!\!
\sum_{\delta(i)\geq1}
\!\!\!
\Pi_{<\frac d2} U^i\Psi_{>-\frac d2}
\Pi_{>-\frac d2} U_i
\big(
\Psi_{>-\frac d2}
A^{-1}
v
\!-\!
\Psi_{-\frac d2}
\big)
\bar1
\!-\!
z^{-1}\Pi_{<\frac d2} D\Psi_{>-\frac d2}
A^{-1}v\bar1
} \\
\displaystyle{
\vphantom{\Big(}
=
z^{-1}
\Pi_{<\frac d2} 
D
\big(
\Psi_{>-\frac d2}
A^{-1}
v
-
\Psi_{-\frac d2}
\big)
\bar1
-
z^{-1}\Pi_{<\frac d2} D\Psi_{>-\frac d2}
A^{-1}v\bar1
} \\
\displaystyle{
\vphantom{\Big(}
=
-z^{-1}
\Pi_{<\frac d2} 
D
\Psi_{-\frac d2}
\bar1
=
w\bar 1
\,,}
\end{array}
\end{equation}
where, for the third equality, we used \eqref{20170626:eq2} 
and the definition \eqref{eq:D} of the shift matrix $D$.
\end{proof}

\subsection{Step 4: proof of Equation \eqref{0330:eq9}}\label{step5}

The operators $A,B$ in \eqref{0330:eq8} lie in 
$\mc RU(\mf g)\otimes\Hom\big(V\big[>-\frac d2\big],V\big[<\frac d2\big]\big)$,
and, by the definition of $B$ and the definition \eqref{0406:eq1}
of the homomorphism $\epsilon:\,\mc RU(\mf g)\to\mb F$,
we have $\epsilon(B)=0$ (where $\epsilon$ here is acting on the first factor
of the tensor product $\mc RU(\mf g)\otimes\Hom\big(V\big[>-\frac d2\big],V\big[<\frac d2\big]\big)$).
It then follows by Proposition \ref{prop:rees3}(f)
that 
$$
\id_{V[<\frac d2]}+BA^{-1}
$$
is an invertible element of 
$\mc R_{\infty}U(\mf g)\otimes\End\big(V\big[<\frac d2\big]\big)$.
Moreover, by Corollary \ref{lem:key3}, we have
$$
(\id+B A^{-1})v\bar1
=(v+w)\bar1
\,.
$$
We then have:
$$
\begin{array}{l}
\displaystyle{
\vphantom{\Big(}
A^{-1}v\bar1
=
A^{-1}(\id+B A^{-1})^{-1}
(\id+B A^{-1})v\bar1
=
(A+B)^{-1}(v+w)\bar1
\,.}
\end{array}
$$
\begin{flushright}
\qedsymbol
\end{flushright}

\subsection{Proof of Theorem \ref{thm:main1}}
\label{sec:4.6}

The proof is similar to the proof of the analogous result 
for classical affine $W$-algebras, presented in \cite[Sec.4]{DSKV18}.
By the Main Lemma \ref{lem:main},
the operator
$|\id_V+z^{-\Delta}U|_{\Psi_{\frac d2},\Pi_{-\frac d2}}$ 
is an invertible element of 
$\mc R_\infty U(\mf g)\otimes\Hom\big(V\big[-\frac d2\big],V\big[\frac d2\big]\big)$,
and equation \eqref{0229:eq3} holds.
Hence, in view of Proposition \ref{0413:prop1},
Theorem \ref{thm:main1} holds provided that
\begin{equation}\label{20170626:eq6}
\big[a,
|\id_V+z^{-\Delta}U|_{\Psi_{\frac d2},\Pi_{-\frac d2}}
\big]\bar 1
=0
\,\,\text{ for all }\,\,
a\in\mf g_{\geq\frac12}
\,.
\end{equation}
By the invertibility of $|\id_V+z^{-\Delta}U|_{\Psi_{\frac d2},\Pi_{-\frac d2}}$
in order to prove equation \eqref{20170626:eq6} it suffices to prove that
\begin{equation}\label{20170626:eq7}
\big[a,
\big(|\id_V+z^{-\Delta}U|_{\Psi_{\frac d2},\Pi_{-\frac d2}}\big)^{-1}
\big]
=0\,.
\end{equation}
By the definition of \eqref{eq:linalg7} of generalized quasideterminant,
we have
\begin{equation}\label{20170626:eq8}
\begin{array}{l}
\displaystyle{
\vphantom{\Big(}
\big[a,
\big(|\id_V+z^{-\Delta}U|_{\Psi_{\frac d2},\Pi_{-\frac d2}}\big)^{-1}
\big]
=
\Pi_{-\frac d2}
\big[a,
(\id_V+z^{-\Delta}U)^{-1}
\big]
\Psi_{\frac d2}
} \\
\displaystyle{
\vphantom{\Big(}
=
-
\Pi_{-\frac d2}
(\id_V+z^{-\Delta}U)^{-1}
\big[a,
z^{-\Delta}U
\big]
(\id_V+z^{-\Delta}U)^{-1}
\Psi_{\frac d2}
\,.}
\end{array}
\end{equation}
Recalling the definition \eqref{20170623:eq10} of the operator $z^{-\Delta}U$, we have
\begin{equation}\label{20170626:eq9}
\begin{array}{l}
\displaystyle{
\vphantom{\Big(}
\big[a,
z^{-\Delta}U
\big]
=
\sum_{i\in I}z^{\delta(i)-1}[a,u_i]U^i
=
\sum_{i,k\in I}z^{\delta(i)-1}([a,u_i]|u^k)u_kU^i
} \\
\displaystyle{
\vphantom{\Big(}
=
\sum_{i,k\in I}z^{\delta(k)-\delta(a)-1}(u_i|[u^k,a])u_kU^i
=
z^{-\delta(a)}\sum_{k\in I}z^{\delta(k)-1}u_k[U^k,\varphi(a)]
} \\
\displaystyle{
\vphantom{\Big(}
=
z^{-\delta(a)}[z^{-\Delta} U,\varphi(a)]
=
z^{-\delta(a)}[\id_V+z^{-\Delta} U,\varphi(a)]
\,.}
\end{array}
\end{equation}
Using \eqref{20170626:eq9},
we can rewrite the RHS of \eqref{20170626:eq8} as
\begin{equation}\label{20170626:eq10}
\begin{array}{l}
\displaystyle{
\vphantom{\Big(}
-
z^{-\delta(a)}
\Pi_{-\frac d2}
(\id_V+z^{-\Delta}U)^{-1}
[\id_V+z^{-\Delta} U,\varphi(a)]
(\id_V+z^{-\Delta}U)^{-1}
\Psi_{\frac d2}
} \\
\displaystyle{
\vphantom{\Big(}
=
-
z^{-\delta(a)}
\Pi_{-\frac d2}
\varphi(a)
(\id_V+z^{-\Delta}U)^{-1}
\Psi_{\frac d2}
+
z^{-\delta(a)}
\Pi_{-\frac d2}
(\id_V+z^{-\Delta}U)^{-1}
\varphi(a)
\Psi_{\frac d2}
\,,}
\end{array}
\end{equation}
Since, by assumption, $a\in\mf g_{\geq\frac12}$,
we have $\varphi(a)\in(\End V)[\geq\frac12]$,
and therefore 
$\Pi_{-\frac d2} \varphi(a)=0$, $\varphi(a)\Psi_{\frac d2}=0$.
Hence, the RHS of \eqref{20170626:eq10} vanishes, proving \eqref{20170626:eq7}.

\begin{flushright}
\qedsymbol
\end{flushright}

\section{\texorpdfstring{$W$}{W}-algebras for classical Lie algebras and the (generalized) Yangian identity}
\label{sec:5}

\subsection{Preliminaries from linear algebra}
\label{sec:5.1}

We review here some linear algebra results which were discussed in \cite{DSKV18}
and which will be needed in the sequel.

Given a vector space $V$ of dimension $N$, 
we denote by $\Omega_V\in\End V\otimes\End V$
the permutation map:
\begin{equation}\label{Omega}
\Omega_V(v_1\otimes v_2)=v_2\otimes v_1
\,\,\text{ for all }\,\,v_1,v_2\in V
\,.
\end{equation}
We shall sometimes write $\Omega_V=\Omega_V^\prime\otimes\Omega_V^{\prime\prime}$
to denote, as usual, a sum of monomials in $\End V\otimes\End V$.
In fact, we can write an explicit formula:
$\Omega=\sum_{i,j=1}^NE_{ij}\otimes E_{ji}$,
where $E_{ij}$ is the ``standard'' basis of $\End V$ consisting of elementary matrices 
w.r.t. any basis of $V$
(obviously, $\Omega$ does not depend on the choice of this basis).
\begin{lemma}[{\cite[Lem.5.1]{DSKV18}}]\label{20170322:lem1}
Let $U$ and $W$ be vector spaces,
and let $A,B\in U\to W$ be linear maps.
We have
\begin{equation}\label{eq:permut}
\Omega_W(A\otimes B)=(B\otimes A)\Omega_U
\end{equation}
\end{lemma}

Let $U$ and $W$ be $M$-dimensional vector spaces
and let $\langle\cdot\,|\,\cdot\rangle:\,W\times U\to\mb F$ 
be a non-degenerate pairing between them.
Let $\{u_k\}_{k=1}^M$ be a basis of $U$ and let $\{w^k\}_{k=1}^M$ be the dual basis of $W$
with respect to $\langle\cdot\,|\,\cdot\rangle$:
$\langle w^k|u_h\rangle=\delta_{h,k}$ for all $h,k=1,\dots,M$.
Recall that we have the following completeness relations:
\begin{equation}\label{eq:completeness}
\sum_{k=1}^M \langle w^k|u\rangle u_k
=
u
\,\,,\,\,\,\,
\sum_{k=1}^M \langle w|u_k\rangle w^k
=w
\,\,\text{ for all } u\in U,\,w\in W
\,.
\end{equation}
For $A\in\End U$, $B\in\End W$, $C\in\Hom(U,W)$, $D\in\Hom(W,U)$, 
we define their adjoints (with respect to $\langle\cdot\,|\,\cdot\rangle$)
$A^\dagger\in\End W$, $B^\dagger\in\End U$, $C^\dagger\in\Hom(U,W)$, $D^\dagger\in\Hom(W,U)$, 
by the formulas ($u,u_1\in U$, $w,w_1\in W$)
\begin{equation}\label{20170704:eq1}
\begin{array}{l}
\displaystyle{
\vphantom{\Big(}
\langle A^\dagger(w)|u\rangle=\langle w|A(u)\rangle
\,,\,\,
\langle w|B^\dagger(u)\rangle=\langle B(w)|u\rangle
\,,} \\
\displaystyle{
\vphantom{\Big(}
\langle C^\dagger(u_1)|u\rangle=\langle C(u)|u_1\rangle
\,,\,\,
\langle w|D^\dagger(w_1)\rangle=\langle w_1|D(w)\rangle
\,.}
\end{array}
\end{equation}
Moreover, it follows from the completeness relations \eqref{eq:completeness}
and the above definition \eqref{20170704:eq1} of adjoints, that 
\begin{equation}\label{eq:abasis}
\begin{array}{l}
\displaystyle{
\vphantom{\Big(}
\sum_{k=1}^Mw^k\otimes A(u_k)=\sum_{k=1}^MA^\dagger(w^k)\otimes u_k
\,,\,\,
\sum_{k=1}^MB(w^k)\otimes u_k=\sum_{k=1}^M w^k\otimes B^\dagger(u_k)
\,,} \\
\displaystyle{
\vphantom{\Big(}
\sum_{k=1}^Mw^k\otimes C(u_k)=\sum_{k=1}^M C^\dagger(u_k)\otimes w^k
\,,\,\,
\sum_{k=1}^MD(w^k)\otimes u_k=\sum_{k=1}^M u_k\otimes D^\dagger(w^k)
\,.}
\end{array}
\end{equation}
We shall denote by $\Omega_U^{\dagger1}$ (resp. $\Omega_U^{\dagger2}$)
the element of $\End W\otimes\End U$ (resp. $\End U\otimes\End W$)
obtained taking the adjoint on the first (resp. second) factor of $\Omega_U$.
Similarly, for $\Omega_W^{\dagger1}\in\End U\otimes\End W$ 
and $\Omega_W^{\dagger2}\in\End W\otimes\End U$.
\begin{lemma}\label{20170322:lem2}
\begin{enumerate}[(a)]
\item
We have 
$\Omega_U^{\dagger2}=\Omega_W^{\dagger1}=\Omega_{U,W}^\dagger$,
where 
\begin{equation}\label{eq:omega+}
\Omega_{U,W}^\dagger(u\otimes w)
:=
\langle w|u\rangle\sum_{k=1}^M u_k\otimes w^k
\,\,,\,\,\,\,
u\in U,\,w\in W
\,.
\end{equation}
Simlarly, we have
$\Omega_U^{\dagger1}=\Omega_W^{\dagger2}=\Omega_{W,U}^\dagger$,
where
\begin{equation}\label{eq:omega+b}
\Omega_{W,U}^\dagger(w\otimes u)
:=
\langle w|u\rangle\sum_{k=1}^M w^k\otimes u_k
\,\,,\,\,\,\,
u\in U,\,w\in W
\,.
\end{equation}
\item
For every $A\in\End U$, we have
\begin{equation}\label{eq:permut-tau}
\begin{array}{l}
\displaystyle{
\vphantom{\Big(}
(A\otimes\id_W)\Omega_{U,W}^\dagger=(\id_U\otimes A^\dagger)\Omega_{U,W}^\dagger
\,\,,\,\,\,\,
\Omega_{U,W}^\dagger(A\otimes\id_W)=\Omega_{U,W}^\dagger(\id_U\otimes A^\dagger)
\,,} \\
\displaystyle{
\vphantom{\Big(}
(A^\dagger\otimes\id_U)\Omega_{W,U}^\dagger=(\id_W\otimes A)\Omega_{W,U}^\dagger
\,\,,\,\,\,\,
\Omega_{W,U}^\dagger(A^\dagger\otimes\id_U)=\Omega_{W,U}^\dagger(\id_W\otimes A)
\,.}
\end{array}
\end{equation}
\end{enumerate}
\end{lemma}
\begin{proof}
If we apply $\Omega_U^{\dagger2}$ to $u\otimes w\in U\otimes W$
and pair the result with $w_1\otimes u_1\in W\otimes U$,
we get
$$
\langle w_1|\Omega_U^{\prime}(u)\rangle\langle(\Omega_U^{\prime\prime})^\dagger(w)|u_1\rangle
=
\langle w_1|\Omega_U^{\prime}(u)\rangle\langle w|\Omega_U^{\prime\prime}(u_1)\rangle
=
\langle w_1|u_1\rangle\langle w|u\rangle
\,,
$$
which is the same result that we get by applying the RHS of \eqref{eq:omega+}
to $u\otimes w$ and pairing it with $w_1\otimes u_1$.
This proves that $\Omega_U^{\dagger2}=\Omega_{U,W}^{\dagger}$.
Similar computations show the remaining identities of part (a).

The four equations \eqref{eq:permut-tau}
are equivalent to 
\begin{equation}\label{20170703:eq1}
\begin{array}{l}
\displaystyle{
\vphantom{\Big(}
(A_1\otimes\id_W)\Omega_{U,W}^\dagger(A_2\otimes\id_W)
=
(\id_U\otimes A_1^\dagger)\Omega_{U,W}^\dagger(\id_U\otimes A_2^\dagger)
\,,} \\
\displaystyle{
\vphantom{\Big(}
(A_1^\dagger\otimes\id_U)\Omega_{W,U}^\dagger(A_2^\dagger\otimes\id_U)
=
(\id_W\otimes A_1)\Omega_{W,U}^\dagger(\id_W\otimes A_2)
\,,}
\end{array}
\end{equation}
where $A_1,A_2\in\End U$.
If we apply the LHS of the first equation in \eqref{20170703:eq1} to $u\otimes w\in U\otimes W$, 
we get, by \eqref{eq:omega+},
\begin{equation}\label{20170703:eq2}
(A_1\otimes\id_W)\Omega_{U,W}^\dagger(A_2(u)\otimes w)
=
\langle w|A_2(u)\rangle \sum_{k=1}^M A_1(u_k)\otimes w^k
\,.
\end{equation}
On the other hand, 
if we apply the RHS of the first equation in \eqref{20170703:eq1} to $u\otimes w$,
we get
$$
(\id_U\otimes A_1^\dagger)\Omega_{U,W}^\dagger(u\otimes A_2^\dagger(w))
=
\langle A_2^\dagger(w)|u\rangle \sum_{k=1}^M u_k\otimes A_1^\dagger(w^k)
\,,
$$
which is the same as \eqref{20170703:eq2} by the definition \eqref{20170704:eq1} of adjoint and the
first identity of \eqref{eq:abasis}.
Similarly for the second equality in \eqref{20170703:eq1}.
\end{proof}

Let $V$ be a vector space of dimension $N$,
with a non-degenerate bilinear form $\langle\cdot\,|\,\cdot\rangle:\,V\times V\to\mb F$,
which is symmetric or skewsymmetric:
\begin{equation}\label{eq:epsilon}
\langle v_1|v_2\rangle=\epsilon\langle v_2|v_1\rangle
\,,\,\,v_1,v_2\in V
\,,\,\,\text{ where }
\epsilon\in\{\pm1\}
\,.
\end{equation}
Let $\{v_k\}_{k=1}^N$ and $\{v^k\}_{k=1}^N$
be dual bases of $V$, i.e.
$\langle v^k|v_h\rangle=\epsilon\langle v_k|v^h\rangle=\delta_{h,k}$.
Let $A^\dagger$ be the adjoint of $A\in\End V$,
i.e. $\langle v_1|A^\dagger(v_2)\rangle=\langle A(v_1)|v_2\rangle$, $v_1,v_2\in V$
(cf. \eqref{20170704:eq1}).
By Lemma \ref{20170322:lem2} 
we have
\begin{lemma}\label{20170322:lem2b}
\begin{enumerate}[(a)]
\item
$\Omega_V^{\dagger1}=\Omega_V^{\dagger2}=:\Omega_V^\dagger$, where
\begin{equation}\label{eq:omega+c}
\Omega_V^\dagger(v_1\otimes v_2)
:=
\langle v_1|v_2\rangle\sum_{k=1}^N v^k\otimes v_k
\,.
\end{equation}
\item
For every $A\in\End V$, we have
\begin{equation}\label{eq:permut-taub}
(A\otimes\id_V)\Omega_{V}^\dagger=(\id_V\otimes A^\dagger)\Omega_{V}^\dagger
\,\,,\,\,\,\,
\Omega_{V}^\dagger(A\otimes\id_V)=\Omega_{V}^\dagger(\id_V\otimes A^\dagger)
\,,
\end{equation}
\end{enumerate}
\end{lemma}
\begin{proof}
It is the same as Lemma \ref{20170322:lem2}  in the special case $U=W=V$.
\end{proof}
Let $U$ and $W$ be vector spaces and let
$\Psi:\,U\hookrightarrow V$ be an injective linear map,
and $\Pi:\,V\twoheadrightarrow W$ be a surjective linear map,
with the property that
\begin{equation}\label{20170704:eq2}
\im\Psi=(\ker\Pi)^{\perp}
\,\,\text{ w.r.t. }\,\,
\langle\cdot\,|\,\cdot\rangle
\,.
\end{equation}
Then, we have an induced non-degenerate pairing
$\langle\cdot\,|\,\cdot\rangle^{\Psi\Pi}:\,W\times U\to\mb F$ given by
\begin{equation}\label{20170704:eq3}
\langle w|u\rangle^{\Psi\Pi}
:=
\langle \Pi^{-1}(w)|\Psi(u)\rangle
\,,\,\,
u\in U,\,w\in W
\,.
\end{equation}
\begin{lemma}\label{20170411:lem2}
For $A\in\End V$, we have
$$
(\Pi A\Psi)^{\dagger}=\epsilon \Pi A^\dagger \Psi
\,\in\Hom(U,W)
\,,
$$
where $\dagger$ in the LHS is w.r.t. the pairing \eqref{20170704:eq3} between $U$ and $W$
(cf. the third equation in \eqref{20170704:eq1}),
while in the RHS is w.r.t. the bilinear form $\langle\cdot\,|\,\cdot\rangle$ of $V$.
\end{lemma}
\begin{proof}
By the definition \eqref{20170704:eq3} of $\langle\cdot\,|\,\cdot\rangle^{\Psi\Pi}$,
and the symmetry condition \eqref{eq:epsilon},
we have
$$
\begin{array}{l}
\displaystyle{
\vphantom{\Big(}
\langle (\Pi A\Psi)^{\dagger}(u_1)|u_2\rangle^{\Psi\Pi}
=
\langle \Pi A\Psi(u_2)|u_1\rangle^{\Psi\Pi}
=
\langle A\Psi(u_2)| \Psi(u_1)\rangle
=
\langle \Psi(u_2)|A^\dagger\Psi(u_1)\rangle
} \\
\displaystyle{
\vphantom{\Big(}
=
\epsilon\langle A^\dagger\Psi(u_1) |\Psi(u_2) \rangle
=
\epsilon\langle \Pi A^\dagger\Psi(u_1) | u_2 \rangle^{\Psi\Pi}
\,.}
\end{array}
$$
\end{proof}
\begin{lemma}\label{20170322:lem5}
If $\{v_k\}_{k=1}^N$, $\{v^k\}_{k=1}^N$ and $\{u_h\}_{h=1}^M$, $\{w^h\}_{h=1}^M$
are dual bases as above,
we have:
\begin{equation}\label{eq:complete2}
\begin{array}{l}
\displaystyle{
\vphantom{\big(}
\sum_{h=1}^Mw^h\otimes \Psi(u_h)
=
\sum_{k=1}^N\Pi(v^k)\otimes v_k
\,\in W\otimes V
\,,} \\
\displaystyle{
\vphantom{\big(}
\sum_{h=1}^M\Psi(u_h)\otimes w^h
=
\epsilon\sum_{k=1}^Nv^k\otimes \Pi(v_k)
\,\in V\otimes W
\,.}
\end{array}
\end{equation}
\end{lemma}
\begin{proof}
Pairing the LHS and the RHS of the first equation in \eqref{eq:complete2} with $u\otimes v\in U\otimes V$, 
we get, by the completeness identity \eqref{eq:completeness},
$$
\sum_{h=1}^M\langle w^h|u\rangle^{\Psi\Pi} \langle \Psi(u_h)|v\rangle
=
\langle \Psi(u)|v\rangle
\,,
$$
and 
$$
\sum_{k=1}^N
\langle \Pi(v^k)|u\rangle^{\Psi\Pi} \langle v_k|v\rangle
=
\epsilon\langle \Pi(v)|u\rangle^{\Psi\Pi}
=
\epsilon\langle v|\Psi(u)\rangle
=
\langle \Psi(u)|v\rangle
\,,
$$
proving the first equation in \eqref{eq:complete2}.
Similarly, if we pair both sides of the second equation in \eqref{eq:complete2} 
with $v\otimes u\in V\otimes U$,
we get
$$
\sum_{h=1}^M\langle\Psi(u_h)|v\rangle\langle w^h|u\rangle^{\Psi\Pi}
=
\langle\Psi(u)|v\rangle
\,,$$
and
$$
\epsilon\sum_{k=1}^N\langle v^k|v\rangle\langle\Pi(v_k)|u\rangle^{\Psi\Pi}
=
\epsilon\langle\Pi(v)|u\rangle^{\Psi\Pi}
=
\epsilon\langle v|\Psi(u)\rangle
=
\langle\Psi(u)|v\rangle
\,.
$$
\end{proof}
\begin{lemma}\label{20170322:lem6}
The following identity holds in $\Hom(V,W)\otimes\Hom(U,V)$:
\begin{equation}\label{permut-omega-dagger1}
(\Pi\otimes\id_V)\Omega_V^\dagger(\id_V\otimes\Psi)
=
(\id_{W}\otimes\Psi)\Omega_{W,U}^\dagger(\Pi\otimes\id_{U})
\,,
\end{equation}
and the following identity holds in $\Hom(U,V)\otimes\Hom(V,W)$:
\begin{equation}\label{permut-omega-dagger2}
(\id_V\otimes\Pi)\Omega_V^\dagger(\Psi\otimes\id_V)
=
(\Psi\otimes\id_{W})\Omega_{U,W}^\dagger(\id_{U}\otimes\Pi)
\,.
\end{equation}
\end{lemma}
\begin{proof}
If we apply the LHS of \eqref{permut-omega-dagger1} to $v\otimes u\in V\otimes U$,
we get
$$
\langle v|\Psi(u)\rangle \sum_{k=1}^N
\Pi(v^k)\otimes v_k\,,
$$
while if we apply the RHS of \eqref{permut-omega-dagger1} to $v\otimes u$,
we get
$$
\langle \Pi(v)|u\rangle^{\Psi\Pi}
\sum_{h=1}^M w^h\otimes \Psi(u_h)
\,.
$$
Hence, equation \eqref{permut-omega-dagger1} follows by the definition \eqref{20170704:eq3}
of the pairing $\langle\cdot\,|\,\cdot\rangle^{\Psi\Pi}$ and by the first equation
in  \eqref{eq:complete2}.
Next, if we apply the LHS of \eqref{permut-omega-dagger2} to $u\otimes v\in U\otimes V$,
we get
$$
\langle \Psi(u)|v\rangle
\sum_{k=1}^Nv^k\otimes \Pi(v_k)
\,,
$$
while, if we apply the RHS of \eqref{permut-omega-dagger2} to $u\otimes v$,
we get
$$
\langle \Pi(v)|u\rangle^{\Psi\Pi}
\sum_{h=1}^M \Psi(u_h)\otimes w^h
\,.
$$
Equation \eqref{permut-omega-dagger2} follows by the definition \eqref{20170704:eq3}
of the pairing $\langle\cdot\,|\,\cdot\rangle^{\Psi\Pi}$ and by the second equation in \eqref{eq:complete2}.
\end{proof}

\subsection{The generalized Yangian identity}
\label{sec:5.2}

Let $\alpha,\beta,\gamma\in\mb F$.
Let $R$ be a unital associative algebra,
and let $U,W$ be $M$-dimensional vector spaces.
For $\beta\neq0$,
we also assume, as in Section \ref{sec:5.1},
that $U$ and $W$ are
endowed with a non-degenerate pairing $\langle\cdot\,|\,\cdot\rangle:\,W\times U\to\mb F$.
As usual, when denoting an element of $R\otimes\Hom(W,U)$
or of  $R\otimes\Hom(W,U)\otimes\Hom(W,U)$,
we omit the tensor product sign on the first factor,
i.e. we treat elements of $R$ as scalars.
\begin{definition}\label{def:gener-yang}
The \emph{generalized} $(\alpha,\beta,\gamma)$-\emph{Yangian identity}
for $A(z)\in R((z^{-1}))\otimes\Hom(W,U)$
is the following identity, holding in $R[[z^{-1},w^{-1}]][z,w]\otimes\Hom(W,U)\otimes\Hom(W,U)$:
\begin{equation}\label{eq:gener-yang}
\begin{array}{l}
\displaystyle{
\vphantom{\Big(}
(z-w+\alpha\Omega_U)
(A(z)\otimes\id_U)
(z+w+\gamma-\beta\Omega_{W,U}^\dagger)
(\id_W\otimes A(w))
} \\
\displaystyle{
\vphantom{\Big(}
=
(\id_U\otimes A(w))
(z+w+\gamma-\beta\Omega_{U,W}^\dagger)
(A(z)\otimes\id_W)
(z-w+\alpha\Omega_W)
\,.}
\end{array}
\end{equation}
\end{definition}
A special case is when $A\in R((z^{-1}))\otimes\End(V)$,
where the vector space $V$
is endowed with a non-degenerate bilinear form $\langle\cdot\,|\,\cdot\rangle$
if $\beta\neq0$,
which we assume to be symmetric or skewsymmetric,
and we let $\epsilon=+1$ and $-1$ respectively.
In this case, the generalized Yangian identity for $A(z)$ reads
\begin{equation}\label{eq:gener-yangV}
\begin{array}{l}
\displaystyle{
\vphantom{\Big(}
(z-w+\alpha\Omega_V)
(A(z)\otimes\id_V)
(z+w+\gamma-\beta\Omega_V^\dagger)
(\id_V\otimes A(w))
} \\
\displaystyle{
\vphantom{\Big(}
=
(\id_V\otimes A(w))
(z+w+\gamma-\beta\Omega_V^\dagger)
(A(z)\otimes\id_V)
(z-w+\alpha\Omega_V)
\,.}
\end{array}
\end{equation}
Using lemmas \ref{20170322:lem1} and \ref{20170322:lem2}(b), equation \eqref{eq:gener-yangV}
can be equivalently rewritten in terms of the following formula
for the commutator 
$[A(z),A(w)]
:=
(A(z)\otimes\id_V)(\id_V\otimes A(w))-(\id_V\otimes A(w))(A(z)\otimes\id_V)$:
\begin{equation}\label{eq:cl-adler}
\begin{array}{l}
\displaystyle{
\vphantom{\Big(}
[A(z),A(w)]
} \\
\displaystyle{
\vphantom{\Big(}
=
\frac{\alpha}{z-w}\Omega_V
\big(
A(w)\otimes A(z)
-
A(z)\otimes A(w)
\big)
} \\
\displaystyle{
\vphantom{\Big(}
-
\frac{\beta}{z\!+\!w\!+\!\gamma}
\big(
(\id_V\!\otimes\! A(w))
\Omega_V^\dagger
(A(z)\!\otimes\!\id_V)
-
(A(z)\!\otimes\!\id_V)
\Omega_V^\dagger
(\id_V\!\otimes\! A(w))
\big)
} \\
\displaystyle{
\vphantom{\Big(}
-
\frac{\epsilon\alpha\beta}{(z\!-\!w)(z\!+\!w\!+\!\gamma)}
\big(
(\id_V\!\otimes\! A(w))
\Omega_V^\dagger
(\id_V\!\otimes\! A(z))
-
(\id_V\!\otimes\! A(z))
\Omega_V^\dagger
(\id_V\!\otimes\! A(w))
\big)
,}
\end{array}
\end{equation}
where we used the identities (cf. \eqref{Omega} and \eqref{eq:omega+c})
$\Omega_V\Omega_V^\dagger=\Omega_V^\dagger\Omega_V=\epsilon\Omega_V^\dagger$.
\begin{remark}\label{20170704:rem1}
For $\gamma=0$, after rescaling $\alpha=\hbar\bar\alpha$ and $\beta=\hbar\bar\beta$,
we can take the classical limit $\hbar\to0$. 
The corresponding Poisson bracket
$\{\cdot\,,\,\cdot\}=\lim_{\hbar\to0}\frac1\hbar[\cdot\,,\,\cdot]$
satisfies:
$$
\begin{array}{l}
\displaystyle{
\vphantom{\Big(}
\{A(z),A(w)\}
} \\
\displaystyle{
\vphantom{\Big(}
=
\frac{\alpha}{z-w}\Omega_V
\big(
A(w)\otimes A(z)
-
A(z)\otimes A(w)
\big)
} \\
\displaystyle{
\vphantom{\Big(}
-
\frac{\beta}{z\!+\!w\!}
\big(
(\id_V\!\otimes\! A(w))
\Omega_V^\dagger
(A(z)\!\otimes\!\id_V)
-
(A(z)\!\otimes\!\id_V)
\Omega_V^\dagger
(\id_V\!\otimes\! A(w))
\big)
\,.}
\end{array}
$$
This equation is the ``finite'' analogue of the generalized Adler identity
introduced in \cite{DSKV18}.
\end{remark}
\begin{remark}\label{20170704:rem2}
In the special case $\alpha=1$, $\beta=\gamma=0$,
equation \eqref{eq:gener-yangV} coincides with the so called RTT presentation
of the Yangian of $\mf{gl}(V)$, cf. \cite{Mol07,DSKV17}. 
(In fact, Molev's presentation corresponds to the choice $\alpha=-1$, $\beta=\gamma=0$,
which we think is less natural.) 
Moreover,
in the special case $\alpha=\beta=\frac12$, $\gamma=0$,
equation \eqref{eq:gener-yangV} coincides with the so called RSRS presentation
of the extended twisted Yangian of $\mf g=\mf{so}(V)$ or $\mf{sp}(V)$, 
depending whether $\epsilon=+1$ or $-1$, cf. \cite{Mol07}. 
(In fact, Molev's presentation corresponds to the choice $\alpha=\beta=-1$, $\gamma=0$,
which we think is less natural.) 
Hence, if $A(z)\in R((z^{-1}))\otimes\End V$
satisfies the generalized $(\frac12,\frac12,0)$-Yangian identity 
we automatically have an algebra homomorphism from 
the extended twisted Yangian $X(\mf g)$ to the algebra $R$.
If, moreover, $A(z)$ satisfies the symmetry condition
(required in the definition of twisted Yangian in \cite{Mol07})
$$
A^\dagger(-z)-\epsilon A(z)
=
-\frac{A(z)-A(-z)}{4z}
\,,
$$
then we have an algebra homomorphism from 
the twisted Yangian $Y(\mf g)$ to the algebra $R$.
\end{remark}

\subsection{The generalized Yangian identity satisfied by the matrix $A(z)$}
\label{sec:5.3}

As in Section \ref{sec:4},
let $\mf g$ be a reductive Lie algebra,
let $\varphi:\,\mf g\to\End V$ be a faithful representation
on the $N$-dimensional vector space $V$,
and let $(\cdot\,|\,\cdot)$ be the associated trace form \eqref{20170317:eq1} of $\mf g$,
which we assume to be non-degenerate.
We denote 
\begin{equation}\label{20170704:eq4}
\Omega^{\mf g}_V
=
\sum_{i\in I}U_i\otimes U^i
\,
\in\End V\otimes\End V
\,,
\end{equation}
where, 
as in Section \ref{sec:4.1},
we let $\{u_i\},\,\{u^i\}$ be dual bases of $\mf g$,
and $\{U_i\}$, $\{U^i\}$
denote the corresponding images in $\End V$.
Note that, 
for $\mf g=\mf{gl}_N$ and $V=\mb F^N$,
we have $\Omega_V^{\mf g}=\Omega_V$,
for $\mf g=\mf{sl}_N$ and $V=\mb F^N$,
we have $\Omega_V^{\mf g}=\Omega_V-\frac1N\id_V\otimes\id_V$,
while for $\mf g=\mf{so}_N$ or $\mf{sp}_N$ and $V=\mb F^N$,
we have $\Omega_V^{\mf g}=\frac12(\Omega_V-\Omega_V^\dagger)$
\cite{DSKV18}.

Consider the operator
$A(z)=z\id_V+\sum_{i\in I}u_iU^i\in U(\mf g)[z]\otimes\End V$ (cf. \eqref{eq:A}).
\begin{lemma}\label{20170704:lem1}
\begin{enumerate}[(a)]
\item
$[A(z),A(w)]=\sum_{i\in I}u_i[\id_V\otimes U^i,\Omega_V^{\mf g}]$.
\item
$\frac{\Omega_V}{z-w}\big(A(w)\otimes A(z)-A(z)\otimes A(w)\big)
=
\sum_{i\in I}u_i[\id_V\otimes U^i,\Omega_V]$.
\item
$(\id_V\otimes A(w))
\frac{\Omega_V^\dagger}{z+w}
(A(z)\otimes\id_V)
-
(A(z)\otimes\id_V)
\frac{\Omega_V^\dagger}{z+w}
(\id_V\otimes A(w))
=
\sum_{i\in I}u_i[\id_V\otimes U^i,\Omega_V^\dagger]$.
\item
$(\id_V\otimes A(w))
\frac{\Omega_V^\dagger}{z-w}
(\id_V\otimes A(z))
-
(\id_V\otimes A(z))
\frac{\Omega_V^\dagger}{z-w}
(\id_V\otimes A(w))
=
\sum_{i\in I}u_i[\id_V\otimes U^i,\Omega_V^\dagger]$.
\end{enumerate}
\end{lemma}
\begin{proof}
We have
$$
[A(z),A(w)]
=
\sum_{i,j\in I}[u_i,u_j]U^i\otimes U^j
=
\sum_{i,k\in I}u_kU^i\otimes [U^k,U_i]
=
\sum_{k\in I}u_k[\id_V\otimes U^k,\Omega_V^{\mf g}]
\,,
$$
proving claim (a).
Next, we have, by \eqref{eq:permut}
$$
\begin{array}{l}
\displaystyle{
\vphantom{\big(}
\frac{\Omega_V}{z-w}\big(A(w)\otimes A(z)-A(z)\otimes A(w)\big)
=
\sum_{i\in I} u_i
\Omega_V
(U^i\otimes\id_V-\id_V\otimes U^i)
} \\
\displaystyle{
\vphantom{\big(}
=
\sum_{i\in I} u_i
[\id_V\otimes U^i,\Omega_V]
\,,}
\end{array}
$$
proving claim (b).
For claim (c) we have
$$
\begin{array}{l}
\displaystyle{
\vphantom{\big(}
(\id_V\otimes A(w))
\frac{\Omega_V^\dagger}{z+w}
(A(z)\otimes\id_V)
-
(A(z)\otimes\id_V)
\frac{\Omega_V^\dagger}{z+w}
(\id_V\otimes A(w))
} \\
\displaystyle{
\vphantom{\big(}
=
\frac1{z+w}
\sum_{i\in I}u_i
\Big(
w
\Omega_V^\dagger
(U^i\otimes\id_V)
+
z
(\id_V\otimes U^i)
\Omega_V^\dagger
} \\
\displaystyle{
\vphantom{\big(}
\,\,\,\,\,\,
-
z
\Omega_V^\dagger
(\id_V\otimes U^i)
-
w
(U^i\otimes\id_V)
\Omega_V^\dagger
\Big)
=
\sum_{i\in I}u_i
[\id_V\otimes U^i,\Omega_V^\dagger]
\,,}
\end{array}
$$
where, for the second equality, we used Lemma \ref{20170322:lem2b}(b)
and the fact that $(U^i)^\dagger=-U^i$.
Finally, we prove claim (d).
By Lemma \ref{20170322:lem2b}(b) and the obvious identity $A^\dagger(z)=-A(-z)$,
we have
$$
\begin{array}{l}
\displaystyle{
\vphantom{\big(}
(\id_V\otimes A(w))
\frac{\Omega_V^\dagger}{z-w}
(\id_V\otimes A(z))
-
(\id_V\otimes A(z))
\frac{\Omega_V^\dagger}{z-w}
(\id_V\otimes A(w))
} \\
\displaystyle{
\vphantom{\big(}
=
-
(\id_V\otimes A(w))
\frac{\Omega_V^\dagger}{z-w}
(A(-z)\otimes\id_V)
+
(A(-z)\otimes\id_V)
\frac{\Omega_V^\dagger}{z-w}
(\id_V\otimes A(w))
\,.}
\end{array}
$$
Hence, claim (d) follows by claim (c), replacing $z$ by $-z$.
\end{proof}
\begin{proposition}\label{20170704:prop}
Let $\mf g$ be one of the classical Lie algebras $\mf{gl}_N$, $\mf{sl}_N$, $\mf{so}_N$ or $\mf{sp}_N$,
and let $V=\mb F^N$ be its standard representation
(endowed, in the cases of $\mf{so}_N$ and $\mf{sp}_N$,
with a non-degenerate symmetric or skewsymmetric bilinear form, respectively).
Then, the operator
$A(z)=z\id_V+\sum_{i\in I}u_iU^i\in U(\mf g)[z]\otimes\End V$
satisfies the generalized Yangian identity \eqref{eq:gener-yangV},
where $\alpha,\beta,\gamma$ are given by the following table:
\begin{equation}\label{table}
\begin{tabular}{ll|lll}
\vphantom{\Big(}
$\mf g$ & $V$ & \,\,$\alpha$\,\, & \,\,$\beta$\,\, & \,\,$\gamma$\,\, \\
\hline 
\vphantom{\Big(}
$\mf{gl}_N$ or $\mf{sl}_N$ & $\mb F^N$ & $1$ & $0$ & $0$ \\
\vphantom{\Big(}
$\mf{so}_N$ or $\mf{sp}_N$ & $\mb F^N$ & $\frac12$ & $\frac12$ & $\frac\epsilon2$ \\
\end{tabular}
\end{equation}
\end{proposition}
\begin{proof}
Recall that the generalized Yangian identity \eqref{eq:gener-yangV}
is equivalent to equation \eqref{eq:cl-adler}.
By Lemma \ref{20170704:lem1}(a),
the LHS of equation \eqref{eq:cl-adler} is 
$$
\sum_{i\in I}u_i
[\id_V\otimes U^i,\Omega_V^{\mf g}]
\,,
$$
while, 
by Lemma \ref{20170704:lem1}(b), (c) and (d),
the RHS of equation \eqref{eq:cl-adler} is
$$
\sum_{i\in I}u_i
\big[\id_V\otimes U^i,
\alpha\Omega_V
-\big(1+\frac{\epsilon\alpha-\gamma}{z+w+\gamma}\big)\beta\Omega_V^\dagger\big]
\,.
$$
The claim follows.
\end{proof}
\begin{remark}\label{20170704:rem3}
We can rescale the values of $\alpha$ and $\beta$ and let $\gamma=0$
at the price of applying an affine transformation $z\mapsto az+b$ 
(cf. Proposition \ref{20170704:prop2}(a)).
For example, in \cite[Sec.2.2]{Mol07} they consider the operator $S(z)$
satisfying the $(-1,-1,0)$-Yangian (=twisted Yangian) identity,
which should thus correspond to $A(-\frac{z}2-\frac{\epsilon}4)$.
(But we think that this choice is less natural.)
\end{remark}

\subsection{The generalized Yangian identity and generalized quasideterminants}
\label{sec:5.4}

\begin{proposition}\label{20170704:prop2}
\begin{enumerate}[(a)]
\item
If $A(z)\in R((z^{-1}))\otimes\Hom(W,U)$ 
satisfies the $(\alpha,\beta,\gamma)$-Yangian identity \eqref{eq:gener-yang},
then $A(az+b)$ satisfies the $(\frac{\alpha}{a},\frac{\beta}{a},\frac{\gamma+2b}{a})$-Yangian identity,
where $a,b\in\mb F$, $a\neq0$.
\item
Suppose that $A(z)\in R((z^{-1}))\otimes\End(V)$ 
satisfies the $(\alpha,\beta,\gamma)$-Yangian identity \eqref{eq:gener-yangV}.
Let $\Psi:\,U\hookrightarrow V$ and $\Pi:\,V\twoheadrightarrow W$
be linear maps satisfying, for $\beta\neq0$, condition \eqref{20170704:eq2}.
Then $\Pi A(z)\Psi\in R((z^{-1}))\otimes\Hom(U,W)$
satisfies the $(\alpha,\beta,\gamma)$-Yangian identity \eqref{eq:gener-yang}
(with $U$ and $W$ exchanged).
\item
Suppose that $A(z)\in R((z^{-1}))\otimes\Hom(U,W)$ 
satisfies the $(\alpha,\beta,\gamma)$-Yangian identity \eqref{eq:gener-yang}
(with $U$ and $W$ exchanged),
and assume that the inverse $A^{-1}(z)$ exists in $R((z^{-1}))\otimes\Hom(W,U)$.
Then $A^{-1}(z)$ satisfies the $(-\alpha,-\beta,\gamma-\beta\dim U)$-Yangian identity.
\item
Suppose that $A(z)\in R((z^{-1}))\otimes\End(V)$ 
satisfies the $(\alpha,\beta,\gamma)$-Yangian identity \eqref{eq:gener-yangV}.
Let $\Psi:\,U\hookrightarrow V$ and $\Pi:\,V\twoheadrightarrow W$
be linear maps satisfying, for $\beta\neq0$, condition \eqref{20170704:eq2}.
Assume that the quasideterminant 
$|A(z)|_{\Psi,\Pi}\in R((z^{-1}))\otimes\Hom(W,U)$ (defined by \eqref{eq:linalg7})
exists.
Then, $|A(az+b)|_{\Psi,\Pi}$
satisfies the $(\frac{\alpha}{a},\frac{\beta}{a},\frac{\gamma-\beta(\dim V-\dim U)+2b}{a})$-Yangian identity.
\end{enumerate}
\end{proposition}
\begin{proof}
Claim (a) is obtained by replacing $z$ by $az+b$ and $w$ by $aw+b$ in equation \eqref{eq:gener-yang}.
For claim (b),
let us compose the LHS of equation \eqref{eq:gener-yangV}
on the left by $\Pi\otimes\Pi$ and on the right by $\Psi\otimes\Psi$.
As a result we get, by Lemma \ref{20170322:lem1} 
and the second equation in Lemma \ref{20170322:lem6},
\begin{equation}\label{20170704:eq5}
\begin{array}{l}
\displaystyle{
\vphantom{\Big(}
(\Pi\otimes\Pi)(z-w+\alpha\Omega_V)
(A(z)\otimes\id_V)
(z+w+\gamma-\beta\Omega_V^\dagger)
(\id_V\otimes A(w))
(\Psi\otimes\Psi)
} \\
\displaystyle{
\vphantom{\Big(}
=
(z-w+\alpha\Omega_W)
(\Pi\otimes\Pi)
(A(z)\otimes\id_V)
(z+w+\gamma-\beta\Omega_V^\dagger)
(\id_V\otimes A(w))
(\Psi\otimes\Psi)
} \\
\displaystyle{
\vphantom{\Big(}
=
(z\!-\!w\!+\!\alpha\Omega_W)
(\Pi A(z)\!\otimes\!\id_W)(\id_V\!\otimes\!\Pi)
(z\!+\!w\!+\!\gamma\!-\!\beta\Omega_V^\dagger)
(\Psi\!\otimes\!\id_V)
(\id_U\!\otimes\! A(w)\Psi)
} \\
\displaystyle{
\vphantom{\Big(}
=
(z\!-\!w\!+\!\alpha\Omega_W)
(\Pi A(z)\!\otimes\!\id_W\!)
(\Psi\!\otimes\!\id_W\!)
(z\!+\!w\!+\!\gamma\!-\!\beta\Omega_{U,W}^\dagger)
(\id_U\!\otimes\!\Pi)
(\id_U\!\otimes\! A(w)\Psi)
} \\
\displaystyle{
\vphantom{\Big(}
=
(z-w+\alpha\Omega_W)
(\Pi A(z)\Psi\otimes\id_W)
(z+w+\gamma-\beta\Omega_{U,W}^\dagger)
(\id_U\otimes \Pi A(w)\Psi)
\,,}
\end{array}
\end{equation}
which is the LHS of the $(\alpha,\beta,\gamma)$-Yangian identity \eqref{eq:gener-yang}
(with $U$ and $W$ exchanged)
for $\Pi A(z)\Psi$.
Similarly for the RHS.
This proves claim (b).
Next, let us prove claim (c).
Since $\Omega_V^2=1$, we have
$$
(z-w+\alpha\Omega_V)^{-1}
=
\frac{z-w-\alpha\Omega_V}{(z-w)^2-\alpha^2}
\,.
$$
Moreover, since  
$(\Omega_{U,W}^\dagger)^2=(\dim U)\Omega_{U,W}^\dagger$
and $(\Omega_{W,U}^\dagger)^2=(\dim U)\Omega_{W,U}^\dagger$
(cf. \eqref{eq:omega+}--\eqref{eq:omega+b}),
we have
$$
(z+w+\gamma-\beta\Omega_{U,W}^\dagger)^{-1}
=
\frac{z+w+\gamma-\beta\dim U+\beta\Omega_{U,W}^\dagger}{(z+w+\gamma)(z+w+\gamma-\beta\dim U)}
\,,
$$
and similarly
$$
(z+w+\gamma-\beta\Omega_{W,U}^\dagger)^{-1}
=
\frac{z+w+\gamma-\beta\dim U+\beta\Omega_{W,U}^\dagger}{(z+w+\gamma)(z+w+\gamma-\beta\dim U)}
\,.
$$
Claim (c) then follows by taking the inverse of both sides of equation \eqref{eq:gener-yang}.
Finally, claim (d) is a consequence of (a), (b) and (c).
\end{proof}

\subsection{The second Main Theorem}
\label{sec:5.5}

\begin{theorem}\label{thm:main2}
Let $\mf g$ be one of the classical Lie algebras $\mf{gl}_N$, $\mf{sl}_N$, $\mf{so}_N$ or $\mf{sp}_N$,
and let $V=\mb F^N$ be its standard representation (endowed, in the cases of $\mf{so}_N$ and $\mf{sp}_N$,
with a non-degenerate symmetric or skewsymmetric bilinear form, respectively,
and we let $\epsilon=+1$ or $-1$ respectively).
Given an $\mf{sl}_2$-triple $(f,h=2x,e)$ in $\mf g$,
consider the quantum finite $W$-algebra $W(\mf g,f)$ 
and the operator $L(z)\in W(\mf g,f)((z^{-1}))\otimes\Hom\big(V\big[-\frac d2\big],V\big[\frac d2\big]\big)$
defined by \eqref{eq:tildeL}-\eqref{eq:L}
(cf. Theorem \ref{thm:main1}).
Then, $L(z)$ satisfies the generalized Yangian identity \eqref{eq:gener-yang}
with the values of $\alpha,\beta,\gamma$ as in the following Table:
\begin{equation}\label{table2}
\begin{tabular}{ll|lll}
\vphantom{\Big(}
$\mf g$ & $V$ & \,\,$\alpha$\,\, & \,\,$\beta$\,\, & \,\,$\gamma$\,\, \\
\hline 
\vphantom{\Big(}
$\mf{gl}_N$ or $\mf{sl}_N$ & $\mb F^N$ & $1$ & $0$ & $0$ \\
\vphantom{\Big(}
$\mf{so}_N$ or $\mf{sp}_N$ & $\mb F^N$ & $\frac12$ & $\frac12$ & $\frac{\epsilon-\dim V+\dim V[\frac d2]}2$ \\
\end{tabular}
\end{equation}
\end{theorem}
\begin{proof}
By Proposition \ref{20170704:prop},
$A(z)=z\id_V+\sum_{i\in I}u_iU^i\in U(\mf g)[z]\otimes\End V$
satisfies the generalized Yangian identity \eqref{eq:gener-yangV}
with $\alpha,\beta,\gamma$ as in Table \eqref{table}.
It then follows by Lemma \ref{lem:X} and the fact that the Yangian identity is graded,
that also $\id_V+z^{-\Delta}U\in\mc RU(\mf g)\otimes\End V$ 
satisfies the generalized Yangian identity with the same values of $\alpha,\beta,\gamma$.
Hence, by Proposition \ref{0410:prop4} and Proposition \ref{20170704:prop2}(d),
$|\id_V+z^{-\Delta}U|_{\Psi_{\frac d2},\Pi_{-\frac d2}}
\in\mc R_{\infty}U(\mf g)\otimes\Hom(V[-\frac d2],V[\frac d2])$  
satisfies the generalized Yangian identity \eqref{eq:gener-yang}
with $\alpha,\beta,\gamma$ as in Table \eqref{table2}.
Recall that the associative product of the $W$-algebra $W(\mf g,f)$
is induced by the product of $U(\mf g)\subset\mc R_\infty U(\mf  g)$.
Recall also that, by Lemma \ref{lem:main},
$L(z)=z^{d+1}|\id_V+z^{-\Delta}U|_{\Psi_{\frac d2},\Pi_{-\frac d2}}\bar 1$.
Hence, applying both sides of the generalized Yangian identity 
for $|\id_V+z^{-\Delta}U|_{\Psi_{\frac d2},\Pi_{-\frac d2}}$ to $\bar 1$,
and multiplying them by $z^{d+1}w^{d+1}$.
we get the generalized Yangian identity
(with the same values of $\alpha,\beta$ and $\gamma$)
for $L(z)\in W(\mf g,f)((z^{-1}))\otimes\Hom(V[-\frac d2],V[\frac d2])$.
\end{proof}
\begin{remark}\label{20170705:rem1}
Recall that, by Remark \ref{20170704:rem2},
the Yangian $Y(\mf{gl}_N)$
is defined by the generalized Yangian identity for $\alpha=1,\beta=\gamma=0$ \cite{Dr86}.
We then automatically have by Theorem \ref{thm:main2}
an algebra homomorphism
$Y(\mf{gl}_n)\to W(\mf{gl}_N,f)$,
where $N=\dim V$ and $n=\dim(V[\frac d2])$,
mapping $T(z)\mapsto L(z)$.
(For this, we need to fix a linear isomorphism $V[-\frac d2]\simeq V[\frac d2]$.)

Next, let $\mf g=\mf{so}_N$ or $\mf{sp}_N$
and $\epsilon=\pm1$ ($+1$ for $\mf{so}_N$ and $-1$ for $\mf{sp}_N$).
Fix also an isomorphism $V[-\frac d2]\simeq V[\frac d2]$,
so that the pairing between them induces a non-degenerate symmetric or skewsymmetric
bilinear form on $V[\frac d2]$,
and let $\bar{\mf g}=\mf{so}_n$ or $\mf{sp}_n$
according to the parity of this bilinear form
(as before, $n=\dim(V[\frac d2])$).
By Remark \ref{20170704:rem2},
the extended twisted Yangian  $X(\bar{\mf g})$ \cite{Mol07}
is defined by the generalized Yangian identity for $\alpha=\beta=\frac12,\gamma=0$.
Hence, by Theorem \ref{thm:main2} and Proposition \ref{20170704:prop2}(a),
we get an algebra homomorphism
$X(\bar{\mf g})\to W(\mf g,f)$,
mapping 
$$
S(z)\mapsto L\big(z+\frac{N-n-\epsilon}{4}\big)
\,.
$$
(Or, for the less natural choice $\alpha=\beta=-1,\gamma=0$ of \cite{Mol07},
$S(z)\mapsto L\big(-\frac z2+\frac{N-n-\epsilon}{4}\big)$.)
\end{remark}
\begin{remark}\label{20170705:rem2}
For $\mf g=\mf{so}_N$ or $\mf{sp}_N$,
the operator $A(z)=z\id_V+\sum_{i\in I}u_iU^i\in U(\mf g)[z]\otimes\End V$
satisfies the skewadjointness condition $A^\dagger(-z)=-A(z)$.
Hence, by Lemma \ref{lem:main} and Lemma \ref{20170411:lem2}
the operator $L(z)\in W(\mf g,f)((z^{-1}))\otimes\Hom(V[-\frac d2],V[\frac d2])$
satisfies the skewadjointness condition
$$
L^\dagger(-z)=-\epsilon L(z)
\,.
$$
The adjointness property of the ``shifted'' operator 
$L\big(z+\frac{N-n-\epsilon}{4}\big)$ is more complicated,
which is reflecting the (more complicated) adjointness property of the operator $S(z)$
defining the twisted Yangian of $\mf{so}_N$ or $\mf{sp}_N$ \cite{Mol07}.
\end{remark}
\begin{remark}\label{20170705:rem3}
Let $p=(p_1^{r_1},\dots,p_s^{r_s})$ be a partition of $N$,
with $p_1>p_2>\dots>p_s$,
and let $r=r_1+\dots+r_s$.
Consider the finite $W$-algebra $W(\mf{gl}_N,f)$, where $f$
is a nilpotent element of $\mf{gl}_N$ associated to the partition $p$.
In \cite{BK06}, Brundan and Kleshchev define a surjective homomorphism 
\begin{equation}\label{20170705:eq1}
\kappa\,:\,\,Y(\mf{gl}_r,\sigma)\,\to\,W(\mf{gl}_N,f)
\,.
\end{equation}
from the shifted Yangian $Y(\mf{gl}_r,\sigma)$ to $W(\mf{gl}_N,f)$.
Recall from \cite{BK06} that the shifted Yangian $Y(\mf{gl}_r,\sigma)$
is generated by the coefficients of the entries of matrices
$D_i(z)\in\Mat_{r_i\times r_i}Y(\mf{gl}_r,\sigma)[[z^{-1}]]$, $i=1,\dots,s$,
$E_i(z)\in\Mat_{r_{i+1}\times r_i}Y(\mf{gl}_r,\sigma)[[z^{-1}]]$, $i=1,\dots,s-1$,
and $F_i(z)\in\Mat_{r_{i}\times r_{i+1}}Y(\mf{gl}_r,\sigma)[[z^{-1}]]$, $i=1,\dots,s-1$,
subject to certain commutation relations.
Then,
once we identify $V[-\frac d2]$ and $V[\frac d2]$ (and fix bases),
we have \cite{Fed17,DSFV18}
$$
L(z)=-(-z)^{p_1}\kappa(D_1(z))
\,,
$$
where $L(z)\in\Mat_{r_1\times r_1} W(\mf{gl}_N,f)((z^{-1}))$ 
is the matrix constructed in Section \ref{sec:4}.
As a consequence, the homomorphism $Y(\mf{gl}_{r_1})\to W(\mf{gl}_N,f)$
described in Remark \ref{20170705:rem1}
is the restriction of the homomorphism \eqref{20170705:eq1}
to the subalgebra $Y(\mf{gl}_{r_1})\subset Y(\mf{gl}_r,\sigma)$
corresponding to $D_1(z)$.
(We thank the anonymous referee of \cite{DSKV17}
for raising this question.)
\end{remark}

\section{An example: rectangular nilpotent element}
\label{sec:7}

Let $N\geq2$ be an integer 
and consider the partition $(p,p,\dots,p)$ of $N$, consisting of $r$ equal parts of size $p$, 
so that $N=rp$.
Let $\mc I=\{(i,h)\in\mb Z_+^2\mid 1\leq i\leq r,1\leq h\leq p\}$, and consider the vector space
$$
V=\bigoplus_{(i,h)\in \mc I}\mb Fv_{(ih)}\cong\mb F^N
\,.
$$
The Lie algebra $\mf g=\mf{gl}(V)\cong\mf{gl}_N$ has a basis consisting of elementary matrices $e_{(i,h)(j,k)}$, where
$(i,h),(j,k)\in \mc I$ (and we denote by $E_{(i,h)(j,k)}$, where
$(i,h),(j,k)\in \mc I$, the same basis when viewed as an element of $\End(V)$).

For any $(i,h)\in \mc I$, we define $(i,h)'=(r+1-i,p+1-h)\in I$.
Moreover, we define $\epsilon_{(i,h)}\in\{\pm1\}$, $(i,h)\in \mc I$, as in one of the following two cases:
\begin{description}
\item[Case 1]
For every $N\geq2$, we let
\begin{equation}\label{20170711:eq1}
\epsilon_{(i,h)}=(-1)^{h+(i-1)p}\,,
\qquad (i,h)\in I
\,.
\end{equation}
\item[Case 2]
For $N=2n$, $n\geq1$, we let
\begin{equation}\label{20170711:eq2}
\epsilon_{(i,h)}=\left\{
\begin{array}{ll}
(-1)^{h+(i-1)p}\,,& 1\leq h+(i-1)p\leq n\,,
\\
(-1)^{1-h+(r+1-i)p}\,,& n+1\leq h+(i-1)p\leq N\,.
\end{array}
\right.
\end{equation}
\end{description}
We define a non-degenerate bilinear form on $V$ as follows:
\begin{equation}\label{20170711:eq3}
\langle v_{(i,h)}|v_{(j,k)}\rangle=-\epsilon_{(i,h)}\delta_{(i,h),(j,k)'}
\,,
\qquad
(i,h),(j,k)\in I
\,.
\end{equation}
It is immediate to check from \eqref{20170711:eq1} and \eqref{20170711:eq2} that we have
$$
\langle v|w\rangle=\epsilon\langle w|v\rangle
\,,\qquad
v,w\in V\,,
$$
where $\epsilon=1$ if we assume $\epsilon_{(i,h)}$ as in Case 2 or as in Case 1 for odd $N$,
and $\epsilon=-1$ if we assume $\epsilon_{(i,h)}$ as in Case 1 for even $N$.

Let $A^\dagger$ denote the adjoint of $A\in\End V$ with respect to \eqref{20170711:eq3}.
Explicitly, in terms of elementary matrices, it is given by:
\begin{equation}\label{eq:dagger}
(E_{(i,h)(j,k)})^\dagger=\epsilon_{(i,h)}\epsilon_{(j,k)}E_{(j,k)'(i,h)'}\,.
\end{equation}
Let 
$$
\mf g_N^\epsilon=\{A\in\End V\,|\,A^\dagger=-A\}=\left\{
\begin{array}{ll}
\mf{so}_N\,,& \epsilon=1\,,
\\
\mf{sp}_N\,,& \epsilon=-1\,.
\end{array}
\right.
$$
For $(i,h)$, $(j,k)\in \mc I$ we define
\begin{equation}\label{eq:F}
F_{(i,h),(j,k)}=E_{(i,h),(j,k)}-\epsilon_{(i,h)}\epsilon_{(j,k)}E_{(j,k)',(i,h)'}\,\,\big(=-F_{(i,h),(j,k)}^\dagger\big)
\,.
\end{equation}
The following commutation relations hold ($(i,h),(j,k),(\alpha,\beta),(\gamma,\delta)\in \mc I$):
\begin{align}\label{comm:BCD}
\begin{split}
&[F_{(i,h),(j,k)},F_{(\alpha,\beta),(\gamma,\delta)}]
=\delta_{(j,k),(\alpha,\beta)}F_{(i,h)(\gamma,\delta)}
-\delta_{(\gamma,\delta),(i,h)}F_{(\alpha,\beta),(j,k)}
\\
&-\epsilon_{(i,h)}\epsilon_{(j,k)}\delta_{(i,h)'(\alpha,\beta)}F_{(j,k)'(\gamma,\delta)}
+\epsilon_{(i,h)}\epsilon_{(j,k)}\delta_{(\gamma,\delta)(j,h)'}F_{(\alpha,\beta)(i,h)'}
\,.
\end{split}
\end{align}
By \eqref{eq:F} the following elements form a basis of $\mf g_N^\epsilon$
$$
\frac{1}{1+\delta_{(i,h),(j,k)'}}f_{(i,h),(j,k)}
:=\frac{1}{1+\delta_{(i,h),(j,k)'}}(e_{(i,h),(j,k)}-\epsilon_{(i,h)}\epsilon_{(j,k)}e_{(j,k)',(i,h)'})
\,,
$$
where $(i,h),(j,k)\in I$, and 
$$
I=\Bigg\{
\begin{array}{l}
\displaystyle{
\vphantom{\Big(}
\big\{(i,h),(j,k)\,\big|\, (1,1)\leq (i,h)\leq (r,p), (1,1)\leq (j,k)\leq (i,h)'\big\}
\,\,\text{ if }\,\, \epsilon=-1
} \\
\displaystyle{
\vphantom{\Big(}
\big\{(i,h),(j,k)\,\big|\, (1,1)\leq (i,h)\leq (r,p), (1,1)\leq (j,k)< (i,h)'\big\}
\,\,\text{ if }\,\, \epsilon=1
\,.}
\end{array}
$$
(We are ordering the indices $(i,h)\in\mc I$ lexicographically.)
Its dual basis, with respect to the trace form \eqref{20170317:eq1}, is 
$$
\big\{\frac12f_{(j,k),(i,h)}\mid (i,h),(j,k)\in I\big\}\,.
$$ 

To the partition $(p,\dots,p)$ we associate the element
$$
f=\sum_{i=1}^r\sum_{h=1}^{p-1}e_{(i,h+1),(i,h)}
\,,
$$
which is a nilpotent element of $\mf g_N^\epsilon$
under the restriction, for $\epsilon=1$, that $r$ is even if $p$ is even.
We can include $f$ in the following $\mf{sl}_2$-triple $\{e,h=2x,f\}\subset\mf g_N^\epsilon$,
where:
\begin{equation}\label{20170711:sl2tripleBCD}
x=\sum_{(i,h)\in\mc I}\frac12(p+1-2h)e_{(i,h),(i,h)}
\,,
\quad
e=
\sum_{i=1}^r\sum_{h=1}^{p-1}h(p-h)e_{(i,h),(i,h+1)}
\,.
\end{equation}
From equation \eqref{20170711:sl2tripleBCD} it follows that the largest $\ad x$-eigenvalue is $d=p-1$ and that
$$
V[\frac{d}{2}]=\bigoplus_{i=1}^r \mb F v_{(i,1)}
\,,
\qquad
V[-\frac{d}{2}]=\bigoplus_{i=1}^r \mb F v_{(i,p)}
\,.
$$
Let $v_i=v_{(i,1)}$, for $i=1,\dots,r$, and let us identify $V[-\frac d2]\stackrel{\sim}{\rightarrow}V[\frac d2]$
via $v_{(i,p)}\mapsto v_i$. Then, the pairing $\langle\cdot|\cdot\rangle^{\Psi_{\frac d2}\Pi_{-\frac d2}}$ associated to the maps
$\Psi_{\frac d2}:V[\frac d2]\hookrightarrow V$ and $\Pi_{-\frac d2}:V\twoheadrightarrow V[-\frac d2]$ (see \eqref{20170704:eq3}) gives a
non-degenerate
bilinear form on $V[\frac d2]$. Using equations \eqref{20170704:eq3} and \eqref{20170711:eq3}, in terms of the basis vectors of $V[\frac d2]$ it reads $(i,j=1,\dots,r$)
\begin{equation}\label{20170711:eq4}
\langle v_i|v_j\rangle^{\Psi_{\frac d2}\Pi_{\frac d2}}
:=\langle v_{(ip)}|v_{(j1)}\rangle=-\epsilon_{(i,p)}\delta_{(i,p)(j,1)'}
\,.
\end{equation}
By equations \eqref{20170711:eq3} and \eqref{20170711:eq1}-\eqref{20170711:eq2} we have
$$
\langle v_j|v_i\rangle^{\Psi_{\frac d2}\Pi_{\frac d2}}
=\phi\langle v_i|v_j\rangle^{\Psi_{\frac d2}\Pi_{\frac d2}}
\,,
\qquad
i,j=1,\dots,r
\,,
$$
where
$$
\phi=\epsilon(-1)^{p-1}
\,.
$$
By Remark \ref{20170705:rem1} we get an algebra homomorphism $\kappa: X(\mf g_r^\phi)\to W(\mf g_N^\epsilon,f)$
from the extended twisted Yangian of $\mf g_r^\phi$ to the $W$-algebra $W(\mf g_{N}^\epsilon,f)$.
This result was first proved by \cite{Br09}.

Finally, we want to give an explicit description of the Lax operator $L(z)\in W(\mf g_N^{\epsilon},f)((z^{-1}))\otimes \Hom(V[-\frac d2],\frac d2)$
defined in \eqref{eq:L}. By identifying $V[-\frac d2]\stackrel{\sim}{\rightarrow}V[\frac d2]$, we can view 
$L(z)$ as an element in
$\Mat_{r\times r}W(\mf g_N^{\epsilon},f)((z^{-1}))$.

Let us further identify 
$$
\Mat_{N\times N}\mb F
\simeq
\Mat_{p\times p}\mb F\otimes\Mat_{r\times r}\mb F
\,,
$$
by mapping $E_{(ih),(jk)}\mapsto E_{hk}\otimes E_{ij}$.
Recalling the explicit expression of the shift matrix $D$ given in \eqref{20170623:eq12} and \eqref{20170623:eq13}, under this identification, we have
\begin{equation*}
\begin{split}
& \id_N\mapsto\id_{p}\otimes\id_{r}
\,\,,\,\,\,\,
D\mapsto \frac12\sum_{h=1}^{p}\left(r(1-h)+\epsilon\delta_{h\geq\frac{p}{2}+1}\right)E_{hh}\otimes \id_{r}
\,,\\
&
F\mapsto \sum_{k=1}^{p-1}E_{k+1,k}\otimes \id_{r}
\,\,,\,\,\,\,
\pi_{\leq\frac12}U\mapsto\sum_{i,j=1}^{r}\sum_{1\leq h\leq k\leq p}\frac1{2c_{(i,h)(j,k)}}f_{(jk),(ih)}E_{hk}\otimes E_{ij}
\,,
\end{split}
\end{equation*}
where $c_{(i,h),(j,k)}=1+\delta_{(i,h)(j,k)'}$, and we denote  (cf. \eqref{eq:dagger})
$f_{(i,h),(j,k)}=e_{(i,h),(j,k)}-\epsilon_{(i,h)}\epsilon_{(j,k)}e_{(j,k)',(i,h)'}$
for every $(i,h)$, $(j,k)\in\mc I$.

The $W$-algebra $W(\mf g_N^\epsilon,f)$ can be identified with a subalgebra of $U(\mf (g_{N}^\epsilon)_{\leq0})$.
By the formula for quasideterminant \eqref{eq:linalg7} and the identification $V[-\frac d2]\stackrel{\sim}{\rightarrow}V[\frac d2]$,
we have that
that $L(z)\in\Mat_{r\times r} U(\mf (g_{N}^\epsilon)_{\leq0})((z^{-1}))$ is defined by
(we use the shorthand notation 
$\tilde f_{(ih),(jk)}=
\frac{1}{2c_{(i,h),(j,k)}}f_{(ih),(jk)}+\frac12\delta_{(ih),(jk)}(r(1-h)+\epsilon\delta_{h\geq\frac{p}{2}+1}$)
\begin{equation}\label{eq:rectangular-quasidet}
L(z)
=
\Big|
(\id_{p}\otimes\id_{r})z
+\sum_{k=1}^{p-1}E_{k+1,k}\otimes\id_{r}
+\sum_{i,j=1}^{r}\sum_{1\leq h\leq k\leq p}\tilde f_{(jk),(ih)}E_{hk}\otimes E_{ij}
\Big|_{I_1J_1}\,,
\end{equation}
where $I_1=\sum_{i=1}^{r}E_{(i1),i}\in\Mat_{N\times r}\mb F$ and
$J_1=\sum_{i=1}^{r}E_{i,(ip)}\in\Mat_{r\times N}\mb F$\, and the quasideterminant \eqref{eq:rectangular-quasidet} can be computed
using the usual formula in \cite[Prop. 4.2]{DSKV16a}. 
As a result we get (see \cite[Sec.9.2]{DSKV17}), for $1\leq i,j\leq r$:
\begin{equation}\label{eq:rectangular-gener-explicit}
\begin{split}
&L_{ij}(z)
=
f_{(jp),(i1)}+
\sum_{s=1}^{p-1}(-1)^s
\sum_{i_1,\dots,i_s=1}^{r}
\sum_{2\leq h_1<\dots<h_s\leq p} \\
& (\delta_{i_1,i}\delta_{h_1-1,1}z+\tilde f_{(i_1,h_1-1),(i1)})
(\delta_{i_2,i_1}\delta_{h_2-1,h_1}z+\tilde f_{(i_2,h_2-1),(i_1h_1)})
\dots \\
& \dots
(\delta_{i_s,i_{s-1}}\delta_{h_s-1,h_{s-1}}z+\tilde f_{(i_s,h_s-1),(i_{s-1},h_{s-1})})
(\delta_{i_s,j}\delta_{p,h_s}z+\tilde f_{(jp),(i_sh_s)})
\,.
\end{split}
\end{equation}
The RHS of \eqref{eq:rectangular-gener-explicit} is a polynomial in $z$, hence it uniquely defines elements
$w_{ji;k}\in W(\mf{g}_N^\epsilon,f)\subset U((\mf g_N^\epsilon)_{\leq0})$, 
$1\leq i,j\leq r$, $0\leq k\leq p-1$, such that
$$
L(z)
=
-\id_{r}(-z)^{p}
+\sum_{k=0}^{p-1}W_k(-z)^k
\,\,,\,\,\,\,
W_k=\big(w_{ji;k}\big)_{i,j=1}^{r}\in\Mat_{r\times r}\!\!\! W(\mf{g}_N^\epsilon,f)
\,.
$$
These elements $w_{ij;k}$ 
generate the $W$-algebra  $W(\mf{g}_N^\epsilon,f)$.
Indeed, by expanding in powers of $z$ the RHS of \eqref{eq:rectangular-gener-explicit}
it is not hard to check that the elements $\gr w_{ij;k}$ span $\mf g^f$ (here, $\gr$ is computed with respect to the Kazhdan filtration
\eqref{0312:eq1}). 
Hence, the homomorphism
$\kappa: X(\mf g_r^\phi)\to W(\mf g_N^\epsilon,f)$ is surjective and we get an isomorphism
$W(\mf{g}_N^\epsilon,f)\cong X(\mf g_r^\phi)/\ker\kappa$,
as showed in \cite{Br09}.

\end{document}